\documentclass[11pt,centertags,oneside]{amsart}

\usepackage{amsmath,amssymb,amsthm,amscd,amstext}
\usepackage[mathscr]{eucal}
\usepackage{mathrsfs}
\usepackage{mathtools}
\usepackage{hyperref}
\usepackage{charter}
\usepackage{enumitem}
\usepackage{xcolor}
\usepackage{fancyhdr}
\usepackage[a4paper,width=16.5cm,top=3cm,bottom=3cm]{geometry}

\numberwithin{equation}{section}

\allowdisplaybreaks
\emergencystretch=\maxdimen
\newtheorem{thm}{Theorem}[section]
\newtheorem{prop}[thm]{Proposition}
\newtheorem{lm}[thm]{Lemma}
\newtheorem{cor}[thm]{Corollary}
\theoremstyle{definition}
\newtheorem{dfn}[thm]{Definition}
\theoremstyle{remark}
\newtheorem{rmk}[thm]{Remark}

\newcommand{\C}{\mathbb{C}}
\newcommand{\N}{\mathbb{N}}
\newcommand{\Q}{\mathbb{Q}}
\newcommand{\Z}{\mathbb{Z}}
\newcommand{\R}{\mathbb{R}}
\newcommand{\Ocal}{\mathcal{O}}
\newcommand{\1}{\mathbf{1}}
\newcommand{\abs}[1]{\left|#1\right|}
\newcommand{\set}[1]{\left\{#1\right\}}
\newcommand{\ip}[1]{\left\langle #1\right\rangle}
\newcommand{\diam}{\operatorname{diam}}
\newcommand{\dist}{\operatorname{dist}}

\definecolor{dgreen}{RGB}{0,100,0}

\title{Borel--Bernstein and Hirst-type Theorems\\
for Nearest-Integer Complex Continued Fractions\\
over Euclidean Imaginary Quadratic Fields}

\author{Kangrae Park}
\address{Department of Mathematical Sciences, Seoul National University}
\email{kangrae.park@snu.ac.kr}

\date{\today}

\begin{document}
\maketitle

\begin{abstract}
For each $d\in\{1,2,3,7,11\}$, let $T_d$ be the nearest-integer complex continued fraction map associated with the Euclidean ring $\Ocal_d$ of imaginary quadratic integers, and let $(a_n)$ denote its digit sequence. We establish two metric results for this five-system family. First, for every sequence $(u_n)_{n\ge1}$ with $u_n\ge1$, the set of points for which $\abs{a_n}\ge u_n$ for infinitely many $n$ has full or zero normalized Lebesgue measure according as $\sum_{n=1}^\infty u_n^{-2}$ diverges or converges. This yields a unified Borel--Bernstein theorem extending the Hurwitz case $d=1$ to all five Euclidean imaginary quadratic fields. Second, if $S\subset\Ocal_d$ is infinite and $\tau(S)$ denotes the convergence exponent of $S$, then the digit-restricted set $F_d(S)=\{z:\ a_n(z)\in S\ \forall n,\ \abs{a_n(z)}\to\infty\}$ satisfies $\dim_H F_d(S)=\tau(S)/2$. More generally, for any cutoff function $f(n)\to\infty$, the restricted set $F_d(S,f)$ is either empty or has the same Hausdorff dimension $\tau(S)/2$. The proof combines quantitative ergodic properties of the nearest-integer systems with a large-digit conformal iterated-function subsystem that is $2$-decaying. We also obtain consequences for sparse patterns, shrinking targets, and almost-sure L\'evy and Khinchine-type laws.
\end{abstract}

\section{Introduction}

The classical Borel--Bernstein theorem gives a sharp zero--one law for the growth of digits
in the regular continued fraction expansion on $(0,1)$.
In the complex setting, a Borel--Bernstein analogue for Hurwitz complex continued fractions
(the case $d=1$) was proved by Gonz\'alez Robert~\cite{GR21}.
For each $d\in\{1,2,3,7,11\}$, let $\Ocal_d$ denote the ring of integers of $\Q(\sqrt{-d})$.

The purpose of this paper is twofold:
\begin{enumerate}[label=(\roman*), leftmargin=3.0em]
\item to prove a Borel--Bernstein theorem for nearest-integer complex continued fractions over
the five Euclidean imaginary quadratic integer rings $\Ocal_d$ with $d\in\{1,2,3,7,11\}$;
\item to prove, in the same five-system setting, a Hirst-type Hausdorff dimension formula
for digit-restricted sets by passing to a large-digit $2$-decaying conformal IFS.
\end{enumerate}

\subsection{Background}

Historically, one can consider the eight nearest-integer algorithms
\[
\mathrm{CF}(1,R),\ \mathrm{CF}(2,R),\ \mathrm{CF}(3,H),\ \mathrm{CF}(3,R),\ \mathrm{CF}(7,H),\ \mathrm{CF}(7,R),\ \mathrm{CF}(11,H),\ \mathrm{CF}(11,R),
\]
where $R$ and $H$ denote the rectangle-type and hexagon-type fundamental domains.
The present paper works with the five-system subfamily
\[
\mathrm{CF}(1,R),\ \mathrm{CF}(2,R),\ \mathrm{CF}(3,H),\ \mathrm{CF}(7,H),\ \mathrm{CF}(11,H),
\]
obtained by choosing one nearest-integer algorithm for each Euclidean imaginary quadratic field.

The earliest sources go back to Hurwitz~\cite{Hurwitz1887}, whose work covers the seed cases
$\mathrm{CF}(1,R)$ and $\mathrm{CF}(3,H)$, and to Lakein~\cite{Lakein1973}, who extended the picture to the remaining Euclidean imaginary quadratic cases.
A different family of complex continued fraction maps was later introduced by Kaneiwa--Shiokawa--Tamura~\cite{KST1976}.
From the ergodic side, Schmidt~\cite{SchmidtCCF1982} gave an early ergodic treatment of complex continued fractions.
In the Gaussian case $\Ocal_1=\Z[i]$, Tanaka~\cite{TanakaCCF1985} studied a complex continued fraction transformation from an ergodic viewpoint, and Nakada~\cite{NakadaGaussian1988} developed the Gaussian-field theory further via natural extensions and metrical applications.
For book-level background on fibred systems and multidimensional/complex continued fractions, see Schweiger~\cite[Chapters~3 and~20]{SchweigerMCF2000}.

For the full eight-system nearest-integer family, Ei--Nakada--Natsui~\cite{ENN23} established the basic ergodic theory,
including an absolutely continuous invariant probability measure and a natural extension.
For the five systems used in the present paper, Kim--Lee--Lim~\cite{KLL25} developed the transfer-operator and spectral framework,
while Baumgartner--Pollicott~\cite{BP25} proved the quantitative dynamical inputs used here:
an invariant probability measure comparable to Lebesgue measure, a sharp $t^{-2}$ tail for the first digit,
and exponential mixing for cylinder sets.

The papers used most directly in the proofs are Baumgartner--Pollicott~\cite{BP25},
for invariant measure, tail asymptotics, and cylinder mixing in the five-system setting,
Kim--Lee--Lim~\cite{KLL25} for the geometric input used in the large-digit IFS construction,
and Nakajima--Takahasi~\cite{NT25} for the abstract dimension theorem for $2$-decaying conformal IFSs.
Gonz\'alez Robert~\cite{GR21} should be viewed here as the Hurwitz-case predecessor and main point of comparison for the Borel--Bernstein problem.

From the viewpoint of regular continued fractions, these complex systems are not merely formal analogues.
They retain the same inversion-and-subtraction mechanism, but the digit alphabet is now the planar lattice $\Ocal_d$ rather than $\N$.
This changes the counting and tail behavior from a one-dimensional problem to a genuinely two-dimensional one,
and it links continued fraction dynamics to the geometry of Euclidean imaginary quadratic fields and, through recent work, to cusp excursions on Bianchi orbifolds.

On the real side, our first theorem belongs to the classical metric theory going back to
Borel and Bernstein~\cite{Borel1898,Bernstein1924,Bugeaud2004}.
Our second theorem lies in the line from Jarn\'{\i}k and Good to Hirst, Cusick, Wang--Wu, and Takahasi
\cite{Jarnik1928,Good1941,Hirst1973,Cusick1978,WangWuHirst,Takahasi2023};
see also \cite{WangWuLarge,HuWu2009,TangZhong2016} for related analogues and extensions in other settings.

On the complex side, modern metrical and dimensional results for Hurwitz complex continued fractions were obtained by
Gonz\'alez Robert~\cite{GR21,GRGood2020}, Nakajima--Takahasi~\cite{NT25}, and Bugeaud--Gonz\'alez Robert--Hussain~\cite{BGRH2025};
see also \cite{Nogueira2001} for a related multidimensional direction.
Here \cite{GRGood2020} gives the Good-type theorem for the Hurwitz expansion, namely that the set of points with $\abs{a_n(z)}\to\infty$ has Hausdorff dimension $1$,
while \cite{NT25} confirms Hirst's conjecture and the restricted slowly growing digit result for the Hurwitz case.
The present paper, however, treats nearest-integer continued fractions uniformly over the five Euclidean imaginary quadratic fields above.

These ingredients leave a natural gap.
On the Borel--Bernstein side, the Hurwitz case $d=1$ was treated in \cite{GR21}, but a unified theorem for the five nearest-integer systems above does not seem to have been written down.
On the Hausdorff-dimension side, the Hurwitz case was treated in \cite{NT25}, and the abstract $2$-decaying conformal IFS theorem of \cite{NT25} is already available,
but it does not appear to have been specialized to this five-system nearest-integer complex continued fraction setting.
The aim of the present paper is to fill these two gaps in a uniform notation.

We also record several further consequences, grouped into mixing-based statements
and large-digit IFS-based statements.

\subsection{Main results}

We fix the notation used in the statements below. For each \(d\in\{1,2,3,7,11\}\), let \(I_d\) denote the nearest-integer fundamental domain, let
\[
T_d:I_d\to I_d
\]
be the associated continued fraction map, and let \(\mathcal N_d\subset I_d\) be the exceptional set of measure zero on which the digit sequence is not globally well-defined. For \(z\in I_d\setminus\mathcal N_d\), write \((a_n(z))_{n\ge1}\) for its digit sequence in \(\mathcal O_d\). Let \(m_d\) denote the normalized Lebesgue measure on \(I_d\), and let \(\mu_d\) denote the absolutely continuous invariant probability measure for \(T_d\). For a finite word \(\mathbf b=(b_1,\dots,b_k)\in\mathcal O_d^k\), write
\[
\ip{\mathbf b}:=\set{z\in I_d\setminus\mathcal N_d:\ a_1(z)=b_1,\dots,a_k(z)=b_k}
\]
for the corresponding cylinder.

Except for Theorem~\ref{thm:BB}, which is naturally stated for $m_d$, the almost-everywhere consequences below will be written with respect to $\mu_d$; by Lemma~\ref{lm:mu_md}, the corresponding $m_d$-formulations are equivalent.

Fix \(d\in\{1,2,3,7,11\}\), and let \((u_n)_{n\ge1}\) satisfy \(u_n\ge1\). Define
\[
E_d(u):=\Bigl\{z\in I_d\setminus\mathcal N_d:\ \abs{a_n(z)}\ge u_n\ \text{for infinitely many }n\in\N\Bigr\}.
\]
We begin with a Borel--Bernstein zero--one law for \(E_d(u)\).

\begin{thm}[Borel--Bernstein]\label{thm:BB}
Fix $d\in\{1,2,3,7,11\}$.
Let $(u_n)_{n\ge1}$ satisfy $u_n\ge 1$.
Then
\[
m_d(E_d(u))=
\begin{cases}
0,& \displaystyle\sum_{n=1}^{\infty}\frac{1}{u_n^{2}}<\infty,\\[1.2ex]
1,& \displaystyle\sum_{n=1}^{\infty}\frac{1}{u_n^{2}}=\infty.
\end{cases}
\]
\end{thm}

Next, for an infinite set $S\subset\Ocal_d$, write $S_*:=S\setminus\{0\}$ and define
\[
\tau(S):=\inf\Bigl\{t>0:\ \sum_{\alpha\in S_*}\abs{\alpha}^{-t}<\infty\Bigr\}.
\]
Thus $\tau(S)=\tau(S_*)$; the value of $\tau(S)$ is insensitive to the presence of $0\in S$. Define
\[
F_d(S):=\Bigl\{z\in I_d\setminus\mathcal N_d:\ a_n(z)\in S\ \forall n,\ \lim_{n\to\infty}\abs{a_n(z)}=\infty\Bigr\}.
\]
Given a function $f:\N\to[1,\infty)$ with $f(n)\to\infty$, set
\[
F_d(S,f):=\Bigl\{z\in F_d(S):\ \abs{a_n(z)}\le f(n)\ \forall n\Bigr\}.
\]

\begin{thm}[Hirst-type dimension formula]\label{thm:HirstMain}
Fix $d\in\{1,2,3,7,11\}$ and the nearest-integer complex continued fraction system $(I_d,T_d)$.
Let $S\subset\Ocal_d$ be infinite and let $f:\N\to[1,\infty)$ satisfy $f(n)\to\infty$.
Then
\[
\dim_H F_d(S)=\frac{\tau(S)}{2}.
\]
Moreover, if $F_d(S,f)\neq \varnothing$ then
\[
\dim_H F_d(S,f)=\frac{\tau(S)}{2}.
\]
\end{thm}

\begin{rmk}
The empty alternative in Theorem~\ref{thm:HirstMain} comes from the admissibility of a finite prefix rather than from the large-digit tail.
Once a single admissible orbit exists, the large-digit full-shift subsystem yields the same Hausdorff dimension $\tau(S)/2$.
For this reason we do not formulate a separate emptiness criterion in terms of $f$ alone.
\end{rmk}

Fix a cutoff function $f:\N\to[1,\infty)$ with $f(n)\to\infty$.
Let $m\ge1$, let $(N_r)_{r\ge1}$ be a strictly increasing sequence of integers, and let
$\mathbf{b}^{(r)}=(b^{(r)}_1,\dots,b^{(r)}_m)\in(\Ocal_d)^m$ satisfy
\[
\min_{1\le j\le m}\abs{b^{(r)}_j}\to\infty
\qquad(r\to\infty),
\]
and
\[
\abs{b^{(r)}_j}\le f(N_r+j-1)
\qquad(r\ge1,\ 1\le j\le m).
\]
Define
\[
F_d(S,f;\mathbf{b}^{(\cdot)},N_\cdot):=
\Bigl\{x\in F_d(S,f):\ (a_{N_r}(x),\dots,a_{N_r+m-1}(x))=\mathbf{b}^{(r)}\ \text{for infinitely many }r\Bigr\}.
\]

\begin{thm}[Sparse patterns]\label{thm:pattern_intro}
Fix $d\in\{1,2,3,7,11\}$ and let $S\subset\Ocal_d$ be infinite.
With the notation above, assume moreover that $b^{(r)}_j\in S$ for all $r\ge1$ and $1\le j\le m$, and that
\[
\liminf_{r\to\infty}\frac{\log\Bigl(\prod_{j=1}^m \abs{b^{(r)}_j}\Bigr)}{N_r}=0.
\]
Then either $F_d(S,f;\mathbf{b}^{(\cdot)},N_\cdot)=\varnothing$ or
\[
\dim_H F_d(S,f;\mathbf{b}^{(\cdot)},N_\cdot)=\frac{\tau(S)}{2}.
\]
\end{thm}

\subsection{Further results}

More generally, the same exponential mixing input yields a shrinking-target theorem for cylinder sets at separated times.

\begin{thm}[Cylinder shrinking targets]\label{thm:cylBC_intro}
Let $(n_j)_{j\ge1}$ be a strictly increasing sequence of nonnegative integers.
For each $j\ge1$ let $\mathbf{b}^{(j)}\in\Ocal_d^{k_j}$ be a finite word of length $k_j\ge1$ and set
\[
C_j:=\ip{\mathbf{b}^{(j)}}.
\]
Assume the separation condition
\(
n_{j+1}\ge n_j+k_j
\) for $j\geq 1$. Define
\[
E_d^{\mathrm{cyl}}:=\Bigl\{z\in I_d:\ T_d^{n_j}(z)\in C_j\ \text{for infinitely many }j\Bigr\}.
\]
Then
\[
\mu_d(E_d^{\mathrm{cyl}})=
\begin{cases}
0,& \displaystyle\sum_{j=1}^{\infty}\mu_d(C_j)<\infty,\\[1.2ex]
1,& \displaystyle\sum_{j=1}^{\infty}\mu_d(C_j)=\infty.
\end{cases}
\]
Moreover, in the divergence case, the corresponding counting function has variance bounded by a constant multiple of its expectation.
\end{thm}

A second almost-sure growth law identifies the exponential rate of the convergent denominators, yielding a L\'evy constant for the five nearest-integer systems.
\begin{thm}[L\'evy constant]\label{thm:levy}
Fix $d\in\{1,2,3,7,11\}$.
For $z\in I_d\setminus\mathcal N_d$, let $P_n(z)/Q_n(z)$ denote the $n$-th convergent of the digit expansion of $z$.
Define
\[
\beta_d:=-\int_{I_d}\log\abs{z}\,d\mu_d(z)\in(0,\infty),
\qquad
L_d:=e^{\beta_d}.
\]
Then for $\mu_d$-a.e.\ $z\in I_d\setminus\mathcal N_d$,
\[
\lim_{n\to\infty}\frac1n\log\abs{Q_n(z)}=\beta_d,
\qquad
\lim_{n\to\infty}\abs{Q_n(z)}^{1/n}=L_d.
\]
\end{thm}

The first-digit tail asymptotic also implies an almost-sure law of large numbers for the geometric mean of the digits.

\begin{thm}[Khinchine exponent]\label{thm:khinchine_intro}
Define
\[
\kappa_d:=\int_{I_d\setminus\mathcal N_d}\log\abs{a_1(z)}\,d\mu_d(z)\in(0,\infty),
\qquad
K_d:=e^{\kappa_d}.
\]
Then for $\mu_d$-a.e.\ $z\in I_d\setminus\mathcal N_d$,
\[
\lim_{n\to\infty}\frac1n\sum_{j=1}^n \log\abs{a_j(z)}=\kappa_d,
\qquad
\lim_{n\to\infty}\Bigl(\prod_{j=1}^n \abs{a_j(z)}\Bigr)^{1/n}=K_d.
\]
\end{thm}

As immediate corollaries of Theorem~\ref{thm:BB}, we obtain sharp growth laws for the digits.

\begin{cor}[Digit-growth log laws]\label{cor:loglaws_intro}
For $m_d$-a.e.\ $z\in I_d$,
\[
\limsup_{n\to\infty}\frac{\log\abs{a_n(z)}}{\log n}=\frac12,
\qquad
\limsup_{n\to\infty}\frac{\log\abs{a_n(z)}-\tfrac12\log n}{\log\log n}=\frac12.
\]
\end{cor}

As an immediate application of Theorem~\ref{thm:HirstMain}, we obtain explicit sample evaluations of $\tau(S)$ and of the corresponding Hausdorff dimensions.

\begin{cor}[Examples for $\tau(S)$ and $\dim_H F_d(S)$]\label{cor:tau_examples_intro}
Fix $d\in\{1,2,3,7,11\}$.
\begin{enumerate}[label=(\roman*), leftmargin=3.0em]
\item If $S=\Ocal_d$, then $\tau(S)=2$ and $\dim_H F_d(S)=1$.
\item If $S=\Z\subset\Ocal_d$, then $\tau(S)=1$ and $\dim_H F_d(S)=1/2$.
\item If $S=\{n^p:\ n\in\N\}\subset\Z$ for an integer $p\ge1$, then $\tau(S)=1/p$ and
\[
\dim_H F_d(S)=\frac{1}{2p}.
\]
\end{enumerate}
\end{cor}

\section{Unified setting and notation}\label{sec:setting}

\subsection{The five Euclidean imaginary quadratic rings and a nearest-integer fundamental domain}

Fix
\[
d\in\{1,2,3,7,11\},
\qquad
\Ocal_d=\text{the ring of integers in }\Q(\sqrt{-d})\subset\C.
\]
Let $\lambda$ denote planar Lebesgue measure on $\C$.
We work with the standard nearest-integer fundamental set $I_d'\subset\C$
(a strict Dirichlet-type polygonal fundamental domain, typically half-open) for translation by $\Ocal_d$
as in \cite{ENN23} such that:
\begin{enumerate}[label=(\roman*), leftmargin=3.0em]
\item $\bigcup_{\alpha\in\Ocal_d}(\alpha+I_d')=\C$;
\item the translates $\{\alpha+\mathrm{int}(I_d')\}_{\alpha\in\Ocal_d}$ are pairwise disjoint;
\item $0\in\mathrm{int}(I_d')$.
\end{enumerate}
Let $I_d:=\overline{I_d'}$.
Write $I_d^\circ:=\mathrm{int}(I_d)=\mathrm{int}(I_d')$.

\begin{lm}\label{lm:boundary_null}
We have $\lambda(\partial I_d')=0$.
\end{lm}
\begin{proof}
In {\cite[Section 3]{BP25}} the set $I_d'$ is given by a Dirichlet-type construction (a polygonal fundamental domain).
In particular, $\partial I_d'$ is a finite union of piecewise $C^1$ curves. Hence $\partial I_d'$ has planar Lebesgue measure $0$.
\end{proof}

\begin{lm}\label{lm:Id_convex}
The set $I_d$ is a compact convex polygon and satisfies $\overline{\mathrm{int}(I_d)}=I_d$.
\end{lm}
\begin{proof}
By the Dirichlet (Voronoi) construction in Baumgartner--Pollicott (2025, \S3)~\cite{BP25}, $I_d'$ is a strict polygonal fundamental domain, and its closure is a compact convex polygon.
Hence $I_d=\overline{I_d'}$ is a compact convex polygon. Since the interior of a convex polygon is dense in its closure, we also have $\overline{\mathrm{int}(I_d)}=I_d$.
\end{proof}

Fix a measurable selection map $[\cdot]:\C\to\Ocal_d$ such that $z-[z]\in I_d'$ for all $z\in\C$;
this choice is unique whenever $z\in \alpha+\mathrm{int}(I_d')$.

\subsection{The Gauss-type map, digits, cylinders, and the exceptional set}

Define $T_d:I_d\to I_d$ by
\[
T_d(z):=
\begin{cases}
\dfrac{1}{z}-\Bigl[\dfrac{1}{z}\Bigr], & z\neq 0,\\[6pt]
0, & z=0,
\end{cases}
\]
where $[\cdot]$ is the fixed measurable selection above.
Let
\[
\mathcal{B}_d:=\C\setminus \bigcup_{\alpha\in\Ocal_d}\bigl(\alpha+\mathrm{int}(I_d')\bigr),
\qquad
B_d:=\set{z\in I_d\setminus\{0\}:\ 1/z\in\mathcal{B}_d}.
\]
Thus $B_d$ is precisely the set of points for which the first digit $[1/z]$ is potentially ambiguous
(i.e.\ $1/z$ lies on the boundary of the partition by translates of $I_d'$).
By Lemma~\ref{lm:boundary_null} and translation-invariance of $\lambda$, the set $\mathcal{B}_d$
is a countable union of translates of $\partial I_d'$, hence $\lambda(\mathcal{B}_d)=0$.
Define the exceptional set
\[
\mathcal{N}_d:=\bigcup_{n\ge0}T_d^{-n}\bigl(B_d\cup\{0\}\bigr).
\]
The digit maps will be used only on $I_d\setminus\mathcal N_d$. For $z\in I_d\setminus\mathcal{N}_d$
define digits $a_n(z)\in\Ocal_d$ by
\[
a_{n+1}(z):=\Bigl[\frac{1}{T_d^n(z)}\Bigr]\qquad(n\ge0).
\]
To view each $a_n$ as a measurable map on all of $I_d$, we extend the definition by setting
$a_n(z):=0$ for all $n\ge1$ and $z\in\mathcal N_d$.

\begin{rmk}[Digit-defined sets are understood on $I_d\setminus\mathcal N_d$]\label{rmk:digit_domain}
The extension of $(a_n)$ to $\mathcal N_d$ is purely conventional and is used only to regard
each digit map as defined on all of $I_d$.
Whenever a set is defined through conditions on the digit sequence
(such as $E_d(u)$, $F_d(S)$, $F_d(S,f)$, and their refinements),
it is understood by definition as a subset of $I_d\setminus\mathcal N_d$.
\end{rmk}
For a finite word $\mathbf{b}=(b_1,\dots,b_k)\in\Ocal_d^k$, define the cylinder
\[
\ip{\mathbf{b}}
:=\ip{b_1,\dots,b_k}
:=\set{z\in I_d\setminus\mathcal N_d:\ a_1(z)=b_1,\dots,a_k(z)=b_k},
\]
well-defined modulo $\mathcal{N}_d$.

\begin{dfn}\label{dfn:admissible}
A finite word $\mathbf{b}\in\Ocal_d^k$ is called \emph{admissible} if $\ip{\mathbf{b}}\neq\varnothing$.
\end{dfn}

Let $m_d$ denote the normalized Lebesgue measure on $I_d$:
\[
m_d(A):=\frac{\lambda(A)}{\lambda(I_d)},\qquad A\subset I_d\ \text{Borel}.
\]

\begin{lm}\label{lm:null_Nd}
We have $\lambda(\mathcal{N}_d)=0$.
\end{lm}
\begin{proof}
By Lemma~\ref{lm:boundary_null} we have $\lambda(\mathcal{B}_d)=0$.
To show $\lambda(B_d)=0$, write
\[
A_k:=\set{z\in I_d:\ 2^{-(k+1)}<\abs{z}\le 2^{-k}}\quad(k\ge 0).
\]
Since $I_d\subset B(0,R_d)$ with $R_d<1$ by Lemma~\ref{lm:Id_ball_ext}, we have
\[
I_d\setminus\{0\}=\bigcup_{k\ge 0}A_k.
\]
Fix $k\ge 0$. On $A_k$ the inversion map $\iota(z):=1/z$ is $C^1$ on a neighborhood of $\overline{A_k}$ and satisfies
\[
\sup_{z\in A_k}\abs{D\iota(z)}=\sup_{z\in A_k}\abs{z}^{-2}\le 2^{2k+2}<\infty.
\]
Moreover,
\[
\iota(A_k)\subset \set{w\in\C:\ 2^k\le \abs{w}<2^{k+1}},
\]
so $\iota^{-1}=\iota$ has bounded derivative on a neighborhood of $\overline{\iota(A_k)}$ as well. Hence $\iota:A_k\to \iota(A_k)$ is bi-Lipschitz, and therefore
\[
\lambda\bigl(\iota^{-1}(\mathcal{B}_d)\cap A_k\bigr)=0.
\]
Summing over $k\ge 0$ yields $\lambda(B_d)=0$. For each $\alpha\in\Ocal_d$ let
\[
U_\alpha:=\set{z\in I_d\setminus\{0\}:\ \Bigl[\frac{1}{z}\Bigr]=\alpha}.
\]
Then $I_d\setminus(\{0\}\cup B_d)=\bigsqcup_{\alpha\in\Ocal_d}U_\alpha$, and on each $U_\alpha$ we have $T_d(z)=1/z-\alpha$, a $C^1$ diffeomorphism.
Hence $T_d$ is nonsingular with respect to $\lambda$ on $I_d\setminus(\{0\}\cup B_d)$, so $\lambda(T_d^{-1}(E))=0$ for every $\lambda$-null set $E\subset I_d$.
Applying this with $E=B_d\cup\{0\}$ and iterating shows $\lambda(T_d^{-n}(B_d\cup\{0\}))=0$ for all $n\ge 0$, and therefore $\lambda(\mathcal N_d)=0$.
\end{proof}

\section{Auxiliary results}\label{sec:external}

We collect here the ingredients for the proofs of the main results.
Along with external ergodic and geometric inputs, we record several auxiliary counting and dimension lemmas used later.

\begin{lm}[{\cite[Lemma 3.1]{BP25}}]\label{lm:acim}
The map $T_d$ admits an ergodic invariant probability measure $\mu_d$ on $I_d$
which is absolutely continuous with respect to $\lambda$.
Moreover, there exists $C'>0$ such that for every Borel set $A\subset I_d$,
\[
\frac{1}{C'}\lambda(A)\ \le\ \mu_d(A)\ \le\ C'\lambda(A).
\]
\end{lm}

\begin{lm}\label{lm:mu_md}
There exists $C''>0$ such that for every Borel set $A\subset I_d$,
\[
\frac{1}{C''}\,m_d(A)\ \le\ \mu_d(A)\ \le\ C''\,m_d(A).
\]
In particular, for Borel $E\subset I_d$ we have $\mu_d(E)=0$ if and only if $m_d(E)=0$, and $\mu_d(E)=1$ if and only if $m_d(E)=1$.
Moreover, for any sequence of Borel sets $(E_n)_{n\ge1}$,
\[
\sum_{n=1}^\infty \mu_d(E_n)<\infty\quad\Longleftrightarrow\quad \sum_{n=1}^\infty m_d(E_n)<\infty.
\]
\end{lm}

\begin{proof}
By Lemma~\ref{lm:acim} and the definition $m_d(A)=\lambda(A)/\lambda(I_d)$, we have
\[
\mu_d(A)\asymp \lambda(A)\asymp m_d(A)
\]
uniformly over Borel $A\subset I_d$, with constants depending only on $d$.
The equivalences for null sets and the series comparison follow.
\end{proof}

\begin{rmk}\label{rmk:mu_md_transfer}
By Lemma~\ref{lm:mu_md}, any $\mu_d$-null set is $m_d$-null and conversely, and likewise for full-measure sets.
Henceforth we state almost-everywhere results with respect to $\mu_d$; the corresponding $m_d$ formulation follows immediately.
\end{rmk}

\begin{lm}[{\cite[Lemma 3.2]{BP25}}]\label{lm:tail}
There exists $H_d>0$ such that
\[
\mu_d\Bigl(\set{z\in I_d\setminus\mathcal N_d:\ \abs{a_1(z)}>t}\Bigr)
\ =\ \frac{H_d}{t^{2}}+O\!\left(\frac{1}{t^{3}}\right)
\qquad (t\to\infty).
\]
\end{lm}

\begin{prop}[{\cite[Proposition 3.1]{BP25}}]\label{prop:mix-cyl}
There exist $\alpha>0$ and $0<\rho<1$ such that for all $k,m\in\N$,
all $\mathbf{b}\in\Ocal_d^k$, all $\mathbf{c}\in\Ocal_d^m$, and all $n\ge0$,
\[
\Bigl|\,
\mu_d\!\Bigl(\ip{\mathbf{b}}\cap T_d^{-(n+k)}\ip{\mathbf{c}}\Bigr)
-\mu_d(\ip{\mathbf{b}})\mu_d(\ip{\mathbf{c}})
\,\Bigr|
\ \le\
\alpha\,\rho^{n}\,\mu_d(\ip{\mathbf{b}})\mu_d(\ip{\mathbf{c}}).
\]
\end{prop}

For $\alpha\in\Ocal_d$, define the inverse branch
\[
h_\alpha(z):=\frac{1}{z+\alpha}.
\]
Then $T_d\circ h_\alpha=\mathrm{id}$ on the appropriate cylinder domain.

\begin{lm}[{\cite[Lemma 2.3]{KLL25}}]\label{lm:Id_ball_ext}
There exists $R_d\in(0,1)$ such that $I_d\subset B(0,R_d)$.
Moreover, we can take $R_d\le \sqrt{15/16}$ for all $d\in\{1,2,3,7,11\}$.
\end{lm}

We use the following terminology from \cite[Section 2.1]{NT25}. Let \(\Delta\subset\C\) be connected and compact with \(\overline{\mathrm{int}(\Delta)}=\Delta\), and assume that \(\Delta\) is convex. Let \(I\subset\C\) be countable, equipped with a size function \(i\mapsto \abs{i}\in[1,\infty)\), and let \(\Phi=\{\phi_i\}_{i\in I}\) be a family of self-maps of \(\Delta\). We call \(\Phi\) a conformal IFS on \(\Delta\) if it satisfies the following conditions:
\begin{enumerate}[label=\textup{(A\arabic*)}, ref=\textup{(A\arabic*)}, leftmargin=3.0em]
\item\label{cond:A1} for distinct \(i,j\in I\), we have
\(
\phi_i(\mathrm{int}(\Delta))\cap \phi_j(\mathrm{int}(\Delta))=\varnothing;
\)
\item\label{cond:A2} there exists a connected open set \(\widetilde\Delta\subset\C\) with \(\Delta\Subset \widetilde\Delta\) such that each \(\phi_i\) extends to a \(C^1\) conformal diffeomorphism
\(
\widetilde\phi_i:\widetilde\Delta\to \widetilde\phi_i(\widetilde\Delta)\subset \widetilde\Delta;
\)
\item\label{cond:A3} there exist \(m\in\N\) and \(\gamma\in(0,1)\) such that for all \((i_1,\dots,i_m)\in I^m\) and all \(z\in\Delta\), we have
\[
\bigl|D(\phi_{i_1}\circ\cdots\circ\phi_{i_m})(z)\bigr|\le \gamma;
\]
\item\label{cond:A4} the set \(\bigcup_{i\in I}\bigl(\phi_i(\Delta)\cap\partial\Delta\bigr)\) is countable.
\end{enumerate}
If \(\# I=\infty\), then \textup{\ref{cond:A3}} yields a coding map \(\pi:I^\N\to\Delta\), defined by
\[
\pi(\omega):=\lim_{n\to\infty}\phi_{\omega_1}\circ\cdots\circ\phi_{\omega_n}(\zeta),
\]
where \(\zeta\in\Delta\) is arbitrary; the limit is independent of \(\zeta\). We write
\[
L(\Phi):=\pi(I^\N),
\qquad
L'(\Phi):=\set{x\in L(\Phi):\pi^{-1}(x)\ \text{is a singleton}}.
\]
Finally, \(\Phi\) is called \(2\)-decaying if it satisfies \textup{\ref{cond:A4}} and there exist constants \(0<C_1\le C_2<\infty\) such that
\[
\frac{C_1}{\abs{i}^2}\le \abs{D\phi_i(z)}\le \frac{C_2}{\abs{i}^2}
\qquad(i\in I,\ z\in\Delta).
\]

\begin{thm}[Consequence of {\cite[Theorem 2.1]{NT25}}]\label{thm:hirst_external}
Let \(\Delta\subset\C\) be connected and compact with \(\overline{\mathrm{int}(\Delta)}=\Delta\), and assume that \(\Delta\) is convex.
Let \(\Phi=\{\phi_i\}_{i\in I}\) be a \(2\)-decaying conformal IFS on \(\Delta\), and let \(\pi:I^\N\to\Delta\) be its coding map.
For an infinite \(S\subset I\), set
\[
\abs{S}:=\set{\abs{i}:i\in S}\subset[1,\infty)
\]
and define
\[
\tau(\abs{S}):=\inf\Bigl\{t>0:\ \sum_{r\in \abs{S}} r^{-t}<\infty\Bigr\}.
\]
Define
\[
F_{\Phi}(S)
:=
\set{x\in L'(\Phi):\ \pi^{-1}(x)=\{\omega(x)\},\ \omega(x)\in S^\N,\ \abs{\omega_n(x)}\to\infty},
\]
and, for \(f:\N\to[1,\infty)\),
\[
F_{\Phi}(S,f)
:=
\set{x\in F_{\Phi}(S):\ \abs{\omega_n(x)}\le f(n)\ \forall n\ge 1}.
\]
Then
\[
\dim_H F_{\Phi}(S)=\frac{\tau(\abs{S})}{2}.
\]
Moreover, if
\[
\set{i\in S:\abs{i}\le f(n)}\neq\varnothing \qquad(n\ge1)
\]
and \(f(n)\to\infty\), then
\[
\dim_H F_{\Phi}(S,f)=\frac{\tau(\abs{S})}{2}.
\]
\end{thm}

We also use a non-autonomous variant of the above notion. By a \emph{non-autonomous conformal IFS} on $\Delta$ we mean a sequence $\boldsymbol{\Phi}=(\Phi_n)_{n\ge1}$, $\Phi_n=\{\phi_i^{(n)}\}_{i\in \mathcal I_n}$, such that the analogues of \textup{\ref{cond:A1}} and \textup{\ref{cond:A2}} hold at each level with the same compact set $\Delta$ and a common extension domain $\widetilde\Delta$, there is a uniform bounded-distortion constant for all finite compositions across levels, and there is a uniform contraction bound for sufficiently long compositions. Its limit set is denoted by $\Lambda(\boldsymbol{\Phi})$, and we call $\boldsymbol{\Phi}$ \emph{subexponentially bounded} if
\[
\lim_{n\to\infty}\frac{1}{n}\log \#\mathcal I_n=0.
\]

\begin{prop}[{\cite[Theorem 3.2]{NT25}}]\label{thm:nonaut_bowen}
Let $\boldsymbol{\Phi}=(\Phi_n)_{n\ge1}$ and $\Phi_n=\{\phi_i^{(n)}\}_{i\in \mathcal I_n}$ be a non-autonomous conformal IFS on a compact convex set $\Delta$ as defined above. Let $\Lambda(\boldsymbol{\Phi})$ be its limit set.
Assume
\[
\lim_{n\to\infty}\frac{1}{n}\log \#\mathcal I_n=0.
\]
For $u\ge0$, define
\[
Z_n^{\boldsymbol{\Phi}}(u):=\sum_{\omega\in \mathcal I_1\times\cdots\times \mathcal I_n}\|D\phi_\omega\|_\infty^{u},
\qquad
P_{\boldsymbol{\Phi}}(u):=\liminf_{n\to\infty}\frac{1}{n}\log Z_n^{\boldsymbol{\Phi}}(u),
\]
where
\(
\phi_\omega:=\phi_{\omega_1}^{(1)}\circ\cdots\circ\phi_{\omega_n}^{(n)}.
\) If $P_{\boldsymbol{\Phi}}(u)\ge 0$, then
\[
\dim_H\Lambda(\boldsymbol{\Phi})\ge u.
\]
\end{prop}

\begin{lm}[{\cite[Lemma 2.3 and 2.4]{NT25}}]\label{lm:NT25_geom}
Let \(\Phi=\{\phi_i\}_{i\in I}\) be a conformal IFS on \(\Delta\), and let \(\widetilde\Delta\) be the connected open set from \textup{\ref{cond:A2}}
Assume there exist \(\zeta\in\Delta\) and \(\delta>0\) such that \(B(\zeta,\delta)\subset \mathrm{int}(\Delta)\).
Then there exist constants \(K_{\mathrm{dist}}=K_{\mathrm{dist}}(\Delta,\widetilde\Delta)\ge1\) and
\(C_{\diam}=C_{\diam}(\Delta,\widetilde\Delta,\zeta,\delta)\ge1\) such that for every \(n\ge1\), every word \(\omega=(i_1,\dots,i_n)\in I^n\), and
\[
\phi_\omega:=\phi_{i_1}\circ\cdots\circ\phi_{i_n},
\]
the following hold:
\begin{enumerate}[label=\textup{(\roman*)}, leftmargin=3.0em]
\item\label{it:koebe_dist}
\(
\sup_{z,w\in\Delta}\dfrac{\abs{D\phi_\omega(z)}}{\abs{D\phi_\omega(w)}}\le K_{\mathrm{dist}}.
\)
\item\label{it:diam_deriv}
\(
C_{\diam}^{-1}\abs{D\phi_\omega(\zeta)}
\le
\diam\bigl(\phi_\omega(\Delta)\bigr)
\le
C_{\diam}\abs{D\phi_\omega(\zeta)}.
\)
\item\label{it:inscribed_disk}
\(
\phi_\omega(\Delta)\ \text{contains a Euclidean disk of radius at least}\
\frac{\delta}{3K_{\mathrm{dist}}}\,\abs{D\phi_\omega(\zeta)}.
\)
\end{enumerate}
\end{lm}

\begin{lm}[{\cite[Lemma 2.5]{NT25}}]\label{lm:coding_countable}
Let \(\Phi=\{\phi_i\}_{i\in I}\) be an infinite conformal IFS on \(\Delta\) satisfying \textup{\ref{cond:A4}}, and let \(\pi:I^\N\to\Delta\) be its coding map.
Then \(L(\Phi)\setminus L'(\Phi)\) is countable.
Equivalently, the set of points in \(\pi(I^\N)\) with more than one address under \(\pi\) is countable.
\end{lm}

\begin{lm}[{\cite[Lemma 2.4(a)]{ENN23}}]\label{lm:denom_growth}
Fix $d\in\{1,2,3,7,11\}$ and let $x\in I_d\setminus\mathcal N_d$.
Let $(Q_n(x))_{n\ge 0}$ be the denominator sequence of convergents associated with the digit expansion of $x$
(equivalently, the bottom-row coefficients in the standard matrix product for the digits).
Then
\[
\abs{Q_n(x)}<\abs{Q_{n+1}(x)}\qquad(n\ge 1).
\]
\end{lm}

\begin{lm}[{\cite[Section~3]{ENN23}}]\label{lm:prefix_mobius_formula}
Let $\mathbf{b}=(b_1,\dots,b_k)\in\Ocal_d^k$ be an admissible finite word, and define
\[
M(\mathbf{b})
:=
\begin{pmatrix}0&1\\ 1&b_1\end{pmatrix}
\cdots
\begin{pmatrix}0&1\\ 1&b_k\end{pmatrix}
=
\begin{pmatrix}P_{k-1}&P_k\\ Q_{k-1}&Q_k\end{pmatrix}.
\]
Then the inverse branch
\[
h_{\mathbf b}:=h_{b_1}\circ\cdots\circ h_{b_k}
\]
is the M\"obius transformation
\[
h_{\mathbf b}(z)=\frac{P_{k-1}z+P_k}{Q_{k-1}z+Q_k}.
\]
In particular, if $x\in \ip{\mathbf b}$, then $(Q_{k-1},Q_k)$ is the denominator pair of the convergents
associated with the prefix $(b_1,\dots,b_k)$ of the digit expansion of $x$.
\end{lm}

\begin{lm}\label{lm:lattice_count}
There exist $C, c_-,c_+>0$, and $R_0\ge 1$ such that:
\begin{enumerate}[label=\textup{(\roman*)}, leftmargin=3.0em]
\item for all $R\ge1$, we have
\(
\#\set{\alpha\in\Ocal_d:\ \abs{\alpha}\le R}\le C R^2;
\)
\item for all $R\ge R_0$,
\(
c_- R^2\le \#\set{\alpha\in\Ocal_d:\ R\le \abs{\alpha}<2R}\le c_+ R^2.
\)
\end{enumerate}
\end{lm}

\begin{proof}
(i) View $\Ocal_d\subset\C\simeq\R^2$ as a full-rank lattice and fix a fundamental parallelogram $P$
of area $\mathrm{area}(P)>0$. For any $R\ge1$, the disk $B(0,R)$ has area $\pi R^2$.
Every translate $\beta+P$ contains at most one lattice point, hence the number of lattice points in $B(0,R)$ is
at most $\mathrm{area}(B(0,R))/\mathrm{area}(P)$ plus a boundary error absorbed into a constant. Thus $\#\{\alpha\in\Ocal_d:\ \abs{\alpha}\le R\}\ll R^2$.\\
(ii) The upper bound follows from (i) by subtracting the count for radius $R$ from that for radius $2R$. For the lower bound, view $\Ocal_d\subset\R^2$ as a full-rank lattice and fix a fundamental parallelogram $P$. For large $R$, the annulus $\set{x\in\R^2:\ R\le |x|<2R}$ contains $\gg R^2$ disjoint translates of $P$, hence contains $\gg R^2$ lattice points.
\end{proof}

\begin{cor}\label{cor:tau_range}
For every infinite set $S\subset\Ocal_d$, we have $0\le \tau(S)\le 2$.
\end{cor}
\begin{proof}
The lower bound is immediate.
Fix $t>2$. By dyadic decomposition and Lemma~\ref{lm:lattice_count}(i),
\[
\sum_{\alpha\in S_*}\abs{\alpha}^{-t}
\le
\sum_{\alpha\in \Ocal_d\setminus\{0\}}\abs{\alpha}^{-t}
\le
\sum_{k\ge0}\ \#\set{\alpha\in\Ocal_d:\ \abs{\alpha}<2^{k+1}}\cdot 2^{-kt}
\ \ll\
\sum_{k\ge0} 2^{2k}\,2^{-kt}
=
\sum_{k\ge0} 2^{k(2-t)}<\infty.
\]
Hence $\tau(S)\le 2$.
\end{proof}

\begin{lm}\label{lm:tau_basic}
Let $S\subset\Ocal_d$ be infinite.
\begin{enumerate}[label=\textup{(\roman*)}, leftmargin=3.0em]
\item If $F\subset\Ocal_d$ is finite, then
\[
\tau(S)=\tau(S\cup F)=\tau(S\setminus F).
\]
\item For every $c>0$ and every $t>0$,
\[
\sum_{\alpha\in S_*}\abs{\alpha}^{-t}<\infty
\quad\Longleftrightarrow\quad
\sum_{\alpha\in S_*}(\abs{\alpha}+c)^{-t}<\infty.
\]
In particular,
\[
\tau(S)=\inf\Bigl\{t>0:\ \sum_{\alpha\in S_*}(\abs{\alpha}+c)^{-t}<\infty\Bigr\}.
\]
\end{enumerate}
\end{lm}

\begin{proof}
    (i) For every $t>0$, the symmetric difference between $S_*$ and $(S\cup F)_*$ is finite, hence
\[
\sum_{\alpha\in S_*}\abs{\alpha}^{-t}<\infty
\quad\Longleftrightarrow\quad
\sum_{\alpha\in (S\cup F)_*}\abs{\alpha}^{-t}<\infty,
\]
and similarly for $S\setminus F$. Taking infimum over $t$ yields the claim. \\
(ii) If $\alpha\in\Ocal_d\setminus\{0\}$ then $\abs{\alpha}^2=N(\alpha)\in\N$, hence $\abs{\alpha}\ge 1$.
Thus
\[
\abs{\alpha}\le \abs{\alpha}+c\le (1+c)\abs{\alpha}\qquad(\alpha\in S_*),
\]
so $(1+c)^{-t}\abs{\alpha}^{-t}\le (\abs{\alpha}+c)^{-t}\le \abs{\alpha}^{-t}$.
Comparison of series gives the equivalence.
\end{proof}

\begin{lm}[{\cite[Theorem 13.12]{Apostol}}]\label{lm:Rd_bound}
For every $\varepsilon>0$ we have
\[
d(n)\ll_\varepsilon n^{\varepsilon}\qquad(n\ge 1),
\]
where $d(n)$ denotes the divisor function.
\end{lm}

\begin{lm}\label{lm:tau_absS}
Let $S\subset\Ocal_d$ be infinite, write $S_*:=S\setminus\{0\}$, and set
\[
\abs{S}:=\set{\abs{\alpha}:\ \alpha\in S_*}\subset[1,\infty).
\]
Define
\[
\tau(\abs{S}):=\inf\Bigl\{t>0:\ \sum_{r\in \abs{S}} r^{-t}<\infty\Bigr\}.
\]
Then $\tau(\abs{S})=\tau(S)$.
\end{lm}
\begin{proof}
Since $\sum_{r\in \abs{S}} r^{-t}\le \sum_{\alpha\in S_*}\abs{\alpha}^{-t}$ for every $t>0$, we have $\tau(\abs{S})\le \tau(S)$.
Fix $t>\tau(\abs{S})$ and choose $\varepsilon>0$ so that $t-2\varepsilon>\tau(\abs{S})$.
Grouping elements of $S_*$ by their modulus, we obtain
\[
\sum_{\alpha\in S_*}\abs{\alpha}^{-t}
\le
\sum_{r\in \abs{S}} \#\set{\alpha\in S_*:\ \abs{\alpha}=r}\, r^{-t}.
\]
For each $r\in \abs{S}$, every $\alpha\in\Ocal_d$ with $\abs{\alpha}=r$ satisfies $(\alpha)(\overline{\alpha})=(r^2)$, so $(\alpha)$ is an ideal divisor of $(r^2)$; since $\Ocal_d$ is a PID with finitely many units, it follows that
\[
\#\set{\alpha\in S_*:\ \abs{\alpha}=r}\ll d(r^2)^2\ll_\varepsilon r^{2\varepsilon},
\]
where the last bound follows from Lemma~\ref{lm:Rd_bound} applied with $\varepsilon/2$.
Hence
\[
\sum_{\alpha\in S_*}\abs{\alpha}^{-t}
\ll_\varepsilon
\sum_{r\in \abs{S}} r^{2\varepsilon}\, r^{-t}
=
\sum_{r\in \abs{S}} r^{-(t-2\varepsilon)}<\infty.
\]
Thus $t>\tau(S)$. Since $t>\tau(\abs{S})$ was arbitrary, we have $\tau(S)\le \tau(\abs{S})$.
\end{proof}

\begin{prop}[Hausdorff--Cantelli lemma, {\cite[Lem 3.10]{BernikDodson1999}}]\label{prop:hausdorff_cantelli}
Let $(X,d)$ be a metric space and let $\{\mathcal U_n\}_{n\ge1}$ be families of subsets of $X$.
Fix $s>0$. If
\[
\sum_{n=1}^{\infty}\ \sum_{U\in\mathcal U_n} \bigl(\diam U\bigr)^s<\infty,
\]
then the limsup set $\limsup_{n\to\infty}\bigcup_{U\in\mathcal U_n} U$ has $\mathcal H^s$-measure $0$.
In particular, its Hausdorff dimension is at most $s$.
\end{prop}

\begin{prop}[Mass distribution principle, {\cite[Theorem 8.17]{Mattila95}}]\label{prop:mdp}
Let $\nu$ be a finite Borel measure on $\C$ and let $E\subset\C$ be Borel with $\nu(E)>0$.
Fix $s>0$. If there exist $C>0$ and $r_0>0$ such that
\[
\nu(B(x,r))\le C r^s \qquad \text{for all }x\in E,\ 0<r<r_0,
\]
then $\dim_H(E)\ge s$.
\end{prop}

\section{Borel--Bernstein and applications}\label{sec:BBproof}
For $t\ge1$ define the first-digit threshold set
\[
A(t):=\set{z\in I_d\setminus\mathcal N_d:\ \abs{a_1(z)}\ge t}.
\]

\begin{cor}\label{cor:t2}
There exist constants $0<c_1\le c_2<\infty$ such that for all $t\ge1$,
\[
\frac{c_1}{t^{2}}
\ \le\
\mu_d\Bigl(\set{z\in I_d\setminus\mathcal N_d:\ \abs{a_1(z)}\ge t}\Bigr)
\ \le\
\frac{c_2}{t^{2}}.
\]
\end{cor}

\begin{proof}
Lemma~\ref{lm:tail} gives $\mu_d(\abs{a_1}>t)=H_d t^{-2}+O(t^{-3})$ for large $t$.
Monotonicity in $t$ and a finite adjustment on $[1,t_0]$ yield the claim.
\end{proof}

\begin{lm}\label{lm:mix-threshold}
With $\alpha,\rho$ as in Proposition~\ref{prop:mix-cyl}, for all $s,t\ge1$ and all $n\ge0$,
\[
\Bigl|\,
\mu_d\bigl(A(s)\cap T_d^{-(n+1)}A(t)\bigr)-\mu_d(A(s))\mu_d(A(t))
\,\Bigr|
\ \le\
\alpha\,\rho^{n}\,\mu_d(A(s))\mu_d(A(t)).
\]
\end{lm}

\begin{proof}
For $B\ge s$ and $C\ge t$ define the finite unions
\[
A_B(s):=\bigcup_{\;s\le \abs{b}\le B}\ip{b},
\qquad
A_C(t):=\bigcup_{\;t\le \abs{c}\le C}\ip{c}.
\]
Then $A_B(s)\uparrow A(s)$ and $A_C(t)\uparrow A(t)$ as $B,C\to\infty$.
For fixed $B,C$, apply Proposition~\ref{prop:mix-cyl} with $k=m=1$ and sum over the finite index sets to obtain
\[
\Bigl|\,
\mu_d\bigl(A_B(s)\cap T_d^{-(n+1)}A_C(t)\bigr)-\mu_d(A_B(s))\mu_d(A_C(t))
\,\Bigr|
\ \le\
\alpha\,\rho^{n}\,\mu_d(A_B(s))\mu_d(A_C(t)).
\]
Letting $B,C\to\infty$ and using continuity from below yields the claim.
\end{proof}

\begin{lm}\label{lm:var_crit}
Let $(X,\mu)$ be a probability space and let $(A_n)_{n\ge1}$ be measurable sets with $p_n:=\mu(A_n)$.
Define
\[
S_N:=\sum_{n=1}^N \1_{A_n},
\qquad
s_N:=\mathbb E_\mu[S_N]=\sum_{n=1}^N p_n.
\]
Assume there exists $C_0<\infty$ such that for every $n\ge1$,
\[
\sum_{k\ge1}\Bigl|\mu(A_n\cap A_{n+k})-p_n p_{n+k}\Bigr|\ \le\ C_0\,p_n.
\]
Then for all $N\ge1$ we have $\mathrm{Var}_\mu(S_N)\le (1+2C_0)\,s_N$.
Moreover, if $\sum_{n=1}^\infty p_n=\infty$ and $(N_j)_{j\ge1}$ satisfies $s_{N_j}\ge j^2$, then
\[
S_{N_j}(x)\ \ge\ \tfrac12 s_{N_j}\ \to\ \infty
\qquad\text{for }\mu\text{-a.e.\ }x.
\]
In particular, $\mu(\limsup_{n\to\infty}A_n)=1$.
\end{lm}

\begin{proof}
Write $Y_n:=\1_{A_n}-p_n$. Then $S_N-s_N=\sum_{n=1}^N Y_n$ and
\[
\mathrm{Var}_\mu(S_N)
=
\sum_{n=1}^N \mathbb E_\mu[Y_n^2]
+
2\sum_{1\le i<j\le N}\mathbb E_\mu[Y_iY_j].
\]
Since $\mathbb E_\mu[Y_n^2]=p_n(1-p_n)\le p_n$ and $\mathbb E_\mu[Y_iY_j]=\mu(A_i\cap A_j)-p_ip_j$, we have
\begin{align*}
\mathrm{Var}_\mu(S_N)
&\le
\sum_{n=1}^N p_n
+
2\sum_{n=1}^N\sum_{k=1}^{N-n}\Bigl|\mu(A_n\cap A_{n+k})-p_np_{n+k}\Bigr|\\
&\le
\sum_{n=1}^N p_n
+
2C_0\sum_{n=1}^N p_n
=
(1+2C_0)\,s_N.
\end{align*}
Assume $\sum_{n=1}^\infty p_n=\infty$ and let $(N_j)_{j\ge1}$ satisfy $s_{N_j}\ge j^2$.
By Chebyshev's inequality and the variance bound,
\[
\mu\Bigl(\bigl|S_{N_j}-s_{N_j}\bigr|\ge \tfrac12 s_{N_j}\Bigr)
\le
\frac{4\,\mathrm{Var}_\mu(S_{N_j})}{s_{N_j}^2}
\le
\frac{4(1+2C_0)}{s_{N_j}}
\le
\frac{4(1+2C_0)}{j^2}.
\]
Since $\sum_{j\ge1} j^{-2}<\infty$, the Borel--Cantelli lemma yields
$|S_{N_j}(x)-s_{N_j}|<\tfrac12 s_{N_j}$ for all sufficiently large $j$, for $\mu$-a.e.\ $x$.
Thus $S_{N_j}(x)\ge \tfrac12 s_{N_j}\to\infty$ for $\mu$-a.e.\ $x$, and in particular $x\in A_n$ for infinitely many $n$.
Hence $\mu(\limsup_{n\to\infty}A_n)=1$.
\end{proof}

\begin{prop}\label{prop:strongBC_sep}
Let $(n_j)_{j\ge1}$ be strictly increasing and let $(k_j)_{j\ge1}\subset\N$.
For each $j$, let $C_j\subset I_d$ be a measurable union of $k_j$-cylinders and assume the separation condition
\[
n_{j+1}\ge n_j+k_j\qquad(j\ge1).
\]
Define events $A_j:=T_d^{-n_j}(C_j)$ and $p_j:=\mu_d(C_j)$.
Assume there exist constants $\alpha>0$ and $\rho\in(0,1)$ such that for all $i<j$,
\[
\bigl|\mu_d(C_i\cap T_d^{-(n_j-n_i)}C_j)-p_ip_j\bigr|
\le
\alpha\,\rho^{\,n_j-n_i-k_i}\,p_ip_j.
\]
Then:
\begin{enumerate}[label=(\roman*), leftmargin=3.0em]
\item If $\sum_{j} p_j<\infty$ then $\mu_d(\limsup_{j\to\infty} A_j)=0$.
\item If $\sum_{j} p_j=\infty$ then $\mu_d(\limsup_{j\to\infty} A_j)=1$ and, with $S_N:=\sum_{j\le N}\1_{A_j}$ and $s_N:=\sum_{j\le N}p_j$, we have $\mathrm{Var}_{\mu_d}(S_N)\ll s_N$.
\end{enumerate}
\end{prop}
\begin{proof}
Part (i) is the (first) Borel--Cantelli lemma.
For (ii), note that for $i<j$,
\[
\mu_d(A_i\cap A_j)=\mu_d(C_i\cap T_d^{-(n_j-n_i)}C_j).
\]
Fix $i$ and sum the mixing estimate over $j>i$.
Since $n_{j+1}\ge n_j+k_j\ge n_j+1$, the integer exponent $n_j-n_i-k_i$ is nonnegative and strictly increasing in $j$, hence
\[
\sum_{j>i}\bigl|\mu_d(A_i\cap A_j)-p_ip_j\bigr|
\le
\alpha p_i\sum_{m\ge 0}\rho^{m}
\le
\frac{\alpha}{1-\rho}p_i.
\]
Thus Lemma~\ref{lm:var_crit} holds and yields $\mu_d(\limsup A_j)=1$ and $\mathrm{Var}_{\mu_d}(S_N)\ll s_N$.
\end{proof}

\begin{proof}[Proof of Theorem~\ref{thm:BB}]
Fix $d\in\{1,2,3,7,11\}$ and set $C_n:=A(u_n)\subset I_d$.
Define
\[
A_n:=T_d^{-(n-1)}(C_n)=\set{z\in I_d\setminus\mathcal N_d:\ \abs{a_n(z)}\ge u_n}.
\]
Then $E_d(u)=\limsup_{n\to\infty}A_n$ modulo $\mathcal{N}_d$ (cf.\ Remark~\ref{rmk:digit_domain}).
By Lemma~\ref{lm:mu_md}, it suffices to prove the zero--one law for $\mu_d$.
Let $p_n:=\mu_d(A_n)=\mu_d(C_n)$ (by invariance).
By Corollary~\ref{cor:t2}, $p_n\asymp u_n^{-2}$.
If $\sum_{n}u_n^{-2}<\infty$, then $\sum_n p_n<\infty$ and the first Borel--Cantelli lemma yields
$\mu_d(\limsup_{n\to\infty}A_n)=0$.
Assume $\sum_{n}u_n^{-2}=\infty$, so $\sum_n p_n=\infty$.
For $i<j$, by invariance,
\[
\mu_d(A_i\cap A_j)
=
\mu_d\bigl(C_i\cap T_d^{-(j-i)}C_j\bigr).
\]
Applying Lemma~\ref{lm:mix-threshold} with $s=u_i$, $t=u_j$ and $n=j-i-1$ gives
\[
\bigl|\mu_d(A_i\cap A_j)-p_ip_j\bigr|
\le
\alpha\,\rho^{\,j-i-1}\,p_ip_j.
\]
Summing over $j>i$ and using $p_j\le 1$, we obtain
\[
\sum_{j>i}\bigl|\mu_d(A_i\cap A_j)-p_ip_j\bigr|
\le
\frac{\alpha}{1-\rho}\,p_i.
\]
Thus Lemma~\ref{lm:var_crit} applies and yields $\mu_d(\limsup_{n\to\infty}A_n)=1$.
Finally, Lemma~\ref{lm:mu_md} transfers the conclusion from $\mu_d$ to $m_d$.
\end{proof}

\begin{proof}[Proof of Theorem~\ref{thm:cylBC_intro}]
For each $j\ge1$ set $C_j:=\ip{\mathbf{b}^{(j)}}$ and $A_j:=T_d^{-n_j}(C_j)$.
Then $C_j$ is a (single) $k_j$-cylinder and the hypothesis $n_{j+1}\ge n_j+k_j$ is exactly the separation condition in Proposition~\ref{prop:strongBC_sep}.
Let $p_j:=\mu_d(C_j)$.
For $i<j$, Proposition~\ref{prop:mix-cyl} gives
\[
\bigl|\mu_d(C_i\cap T_d^{-(n_j-n_i)}C_j)-p_ip_j\bigr|
\le
\alpha\,\rho^{\,n_j-n_i-k_i}\,p_ip_j.
\]
Therefore Proposition~\ref{prop:strongBC_sep} applies.
If $\sum_j p_j<\infty$ then $\mu_d(\limsup A_j)=0$, while if $\sum_j p_j=\infty$ then $\mu_d(\limsup A_j)=1$ and $\mathrm{Var}_{\mu_d}(S_N)\ll \sum_{j\le N}p_j$.
\end{proof}

\begin{proof}[Proof of Theorem~\ref{thm:levy}]
Set
\[
\varphi(z):=
\begin{cases}
-\log\abs{z}, & z\neq 0,\\
0, & z=0.
\end{cases}
\]
By Lemma~\ref{lm:mu_md}, we have $\mu_d\asymp\lambda$ on $I_d$.
Moreover,
\[
\int_{B(0,\varepsilon)}-\log\abs{z}\,d\lambda(z)
=
2\pi\int_0^\varepsilon r(-\log r)\,dr
<\infty,
\]
so $\varphi\in L^1(\mu_d)$.
By Lemma~\ref{lm:Id_ball_ext}, $I_d\subset B(0,R_d)$ for some $R_d\in(0,1)$, and hence
$\varphi(z)\ge -\log R_d>0$ for $z\in I_d\setminus\{0\}$. Since $\mu_d(\{0\})=0$, it follows that
\[
\beta_d=\int_{I_d}\varphi\,d\mu_d\in(0,\infty).
\]
Fix $z\in I_d\setminus\mathcal N_d$ and write $z_n:=T_d^n(z)$.
Then $z_n\neq 0$ for all $n\ge 0$.
Let $(Q_n(z))_{n\ge 0}$ be the denominator sequence, with
\(
Q_n(z)=a_n(z)Q_{n-1}(z)+Q_{n-2}(z)\qquad(n\ge 1),
\) and initial data $Q_{-1}(z)=0$, $Q_0(z)=1$.
Set
\(
X_n:=Q_{n-1}(z)z_n+Q_n(z).
\) Since $z_{n-1}^{-1}=a_n(z)+z_n$, we have
\[
X_n
=
Q_{n-1}(z)\bigl(z_n+a_n(z)\bigr)+Q_{n-2}(z)
=
z_{n-1}^{-1}\bigl(Q_{n-2}(z)z_{n-1}+Q_{n-1}(z)\bigr)
=
z_{n-1}^{-1}X_{n-1}.
\]
Moreover,
\(
X_1=z_1+a_1(z)=z^{-1}.
\) Hence,
\(
X_n=\prod_{j=0}^{n-1} z_j^{-1}
\) by induction and therefore
\[
\log\abs{Q_{n-1}(z)T_d^n(z)+Q_n(z)}
=
\sum_{j=0}^{n-1}\varphi(T_d^j(z)).
\]
By Lemma~\ref{lm:denom_growth}, $\abs{Q_{n-1}(z)}<\abs{Q_n(z)}$, while $\abs{T_d^n(z)}\le R_d<1$ by Lemma~\ref{lm:Id_ball_ext}.
Thus
\[
(1-R_d)\abs{Q_n(z)}
<
\abs{Q_{n-1}(z)T_d^n(z)+Q_n(z)}
<
2\abs{Q_n(z)}.
\]
We deduce that
\[
\log\abs{Q_n(z)}
=
\sum_{j=0}^{n-1}\varphi(T_d^j(z))+O(1).
\]
Dividing by $n$ and applying the Birkhoff ergodic theorem to $\varphi\in L^1(\mu_d)$, we obtain
\[
\frac1n\log\abs{Q_n(z)}\longrightarrow \int_{I_d}\varphi\,d\mu_d=\beta_d
\qquad
\text{for }\mu_d\text{-a.e.\ }z.
\]
Exponentiating gives the second assertion, and the $m_d$-a.e.\ statement follows from Lemma~\ref{lm:mu_md}.
\end{proof}

\begin{proof}[Proof of Theorem~\ref{thm:khinchine_intro}]
Define
\[
f(z):=
\begin{cases}
\log\abs{a_1(z)},& z\in I_d\setminus\mathcal N_d,\\
0,& z\in \mathcal N_d.
\end{cases}
\]
By Lemma~\ref{lm:null_Nd} and Lemma~\ref{lm:mu_md}, we have $\mu_d(\mathcal N_d)=0$, and therefore
\[
\kappa_d=\int_{I_d} f\,d\mu_d.
\]
Moreover $f\ge 0$, since $a_1(z)\in\Ocal_d\setminus\{0\}$ on $I_d\setminus\mathcal N_d$ implies $\abs{a_1(z)}\ge 1$.
By Lemma~\ref{lm:tail}, there exist $t_0\ge 1$ and $C_1>0$ such that
\[
\mu_d\bigl(\abs{a_1}>t\bigr)\le C_1 t^{-2}
\qquad(t\ge t_0).
\]
For $u>0$, since $f=0$ on $\mathcal N_d$, we have
\[
\mu_d(f>u)=\mu_d\bigl(\{z\in I_d\setminus\mathcal N_d:\ \abs{a_1(z)}>e^u\}\bigr).
\]
Using the layer-cake formula, we obtain
\[
\int_{I_d} f\,d\mu_d
=
\int_0^\infty \mu_d(f>u)\,du
\le
\int_0^{\log t_0} 1\,du
+
\int_{\log t_0}^\infty C_1 e^{-2u}\,du
<\infty.
\]
Thus $f\in L^1(\mu_d)$ and $\kappa_d<\infty$.
Moreover, since $H_d>0$ in Lemma~\ref{lm:tail}, we have $\mu_d(\abs{a_1}>t)>0$ for some $t>1$, and hence
\[
\kappa_d
\ge
(\log t)\mu_d(\abs{a_1}>t)
>0.
\]
Since $\mu_d$ is $T_d$-invariant and ergodic by Lemma~\ref{lm:acim}, the Birkhoff ergodic theorem gives, for $\mu_d$-a.e.\ $z\in I_d$,
\[
\lim_{n\to\infty}\frac1n\sum_{j=0}^{n-1} f(T_d^j z)
=
\int_{I_d} f\,d\mu_d
=
\kappa_d.
\]
If $z\in I_d\setminus\mathcal N_d$, then $T_d^j(z)\in I_d\setminus\mathcal N_d$ for every $j\ge 0$, and by definition of the digits,
\[
f(T_d^j z)=\log\abs{a_1(T_d^j z)}=\log\abs{a_{j+1}(z)}.
\]
Therefore, for $\mu_d$-a.e.\ $z\in I_d\setminus\mathcal N_d$,
\[
\lim_{n\to\infty}\frac1n\sum_{j=1}^n \log\abs{a_j(z)}=\kappa_d.
\]
Exponentiating gives the geometric-mean limit with $K_d=e^{\kappa_d}$.
\end{proof}

\begin{proof}[Proof of Corollary~\ref{cor:loglaws_intro}]
Fix $\alpha>0$ and apply Theorem~\ref{thm:BB} with $u_n=n^{\alpha}$. If $\alpha>1/2$, then $\sum_n u_n^{-2}<\infty$, so $\abs{a_n(z)}\ge n^{\alpha}$ holds only finitely often for $m_d$-a.e.\ $z$, and hence
$\limsup_{n\to\infty}\frac{\log\abs{a_n(z)}}{\log n}\le \alpha$.
If $0<\alpha<1/2$, then $\sum_n u_n^{-2}=\infty$, so $\abs{a_n(z)}\ge n^{\alpha}$ holds infinitely often for $m_d$-a.e.\ $z$, and hence
$\limsup_{n\to\infty}\frac{\log\abs{a_n(z)}}{\log n}\ge \alpha$.
Letting $\alpha\downarrow 1/2$ in the first inequality and $\alpha\uparrow 1/2$ in the second gives
\[
\limsup_{n\to\infty}\frac{\log\abs{a_n(z)}}{\log n}=\frac12.
\]
Fix $c\in\R$ and apply Theorem~\ref{thm:BB} with $u_n=\sqrt{n}(\log n)^c$ for $n\ge 3$.
Then $\sum_n u_n^{-2}$ converges if and only if $c>1/2$.
If $c>1/2$, then $\abs{a_n(z)}\ge \sqrt{n}(\log n)^c$ occurs only finitely often for $m_d$-a.e.\ $z$, and hence
$\limsup_{n\to\infty}\frac{\log \abs{a_n(z)}-\tfrac12\log n}{\log\log n}\le c$.
If $c<1/2$, then the same event occurs infinitely often for $m_d$-a.e.\ $z$, and hence
$\limsup_{n\to\infty}\frac{\log \abs{a_n(z)}-\tfrac12\log n}{\log\log n}\ge c$.
Letting $c\downarrow 1/2$ and $c\uparrow 1/2$ yields
\[
\limsup_{n\to\infty}\frac{\log \abs{a_n(z)}-\tfrac12\log n}{\log\log n}=\frac12.
\]
\end{proof}

\section{Hirst-type dimension formula}\label{sec:HirstProof}

For $A\ge 1$ set
\[
\Ocal_d^{\ge A}:=\set{\alpha\in\Ocal_d:\ \abs{\alpha}\ge A}.
\]
Recall that $I_d^\circ:=\mathrm{int}(I_d)$.

\begin{lm}[Geometry of the fundamental domain]\label{lm:geometry_fd}
Let
\[
r_d:=\dist\bigl(0,\C\setminus I_d^\circ\bigr),
\]
and let $R_d\in(0,1)$ be as in Lemma~\ref{lm:Id_ball_ext}. Then
\[
B(0,r_d)\subset I_d^\circ \subset I_d \subset B(0,R_d),
\]
and in particular
\[
0<r_d\le R_d<1.
\]
\end{lm}
\begin{proof}
Since $0\in I_d^\circ$ and $\C\setminus I_d^\circ$ is closed, we have $r_d>0$.
If $\abs{z}<r_d$, then $\dist(z,\C\setminus I_d^\circ)\ge r_d-\abs{z}>0$, hence $z\notin\C\setminus I_d^\circ$ and so $z\in I_d^\circ$.
The inclusion $I_d\subset B(0,R_d)$ is Lemma~\ref{lm:Id_ball_ext}.
Therefore
\[
B(0,r_d)\subset I_d^\circ \subset I_d \subset B(0,R_d),
\]
and the inequality $0<r_d\le R_d<1$ follows immediately.
\end{proof}

\begin{lm}\label{lm:trapping}
Set $\Delta_d:=I_d$ and $\widetilde\Delta_d:=B(0,2R_d)$.
Let
\[
K_{\mathrm{dist}}=K_{\mathrm{dist}}(\Delta_d,\widetilde\Delta_d)\ge 1,
\qquad
C_{\diam}=C_{\diam}(\Delta_d,\widetilde\Delta_d,0,r_d)\ge 1
\]
be the constants from Lemma~\ref{lm:NT25_geom} for the compact inclusion $\Delta_d\Subset \widetilde\Delta_d$.
Then there exists $A_0\ge 1$ such that for every $\alpha\in\Ocal_d^{\ge A_0}$ we obtain:
\begin{enumerate}[label=(\roman*), leftmargin=3.0em]
\item $h_\alpha(I_d)\subset B(0,r_d)\subset I_d^\circ$;
\item $h_\alpha(I_d)\cap\partial I_d=\varnothing$;
\item if $\alpha\neq\beta$ then $h_\alpha(I_d^\circ)\cap h_\beta(I_d^\circ)=\varnothing$;
\item $h_\alpha$ extends to a $C^1$ conformal diffeomorphism
$\widetilde h_\alpha:\widetilde\Delta_d\to \widetilde h_\alpha(\widetilde\Delta_d)\subset\widetilde\Delta_d$;
\item $\sup_{z\in I_d}\abs{D h_\alpha(z)}\le \gamma$ where $\gamma:=1/(A_0-R_d)^2$ satisfies $0<\gamma<1$ and
\[
C_{\diam}^{2}K_{\mathrm{dist}}\gamma\le \tfrac12.
\]
\end{enumerate}
By Lemma~\ref{lm:Id_convex}, the compact set $\Delta_d:=I_d$ is connected, convex, and satisfies $\overline{\mathrm{int}(\Delta_d)}=\Delta_d$, and
\[
\Phi_d^{\ge A_0}:=\{h_\alpha\}_{\alpha\in\Ocal_d^{\ge A_0}}
\]
is a conformal IFS on $\Delta_d$ satisfying \textup{\ref{cond:A1}}--\textup{\ref{cond:A4}}; moreover, the union in \textup{\ref{cond:A4}} is empty.
\end{lm}
\begin{proof}
Choose $A_0$ large so that
\[
A_0>R_d+\frac{1}{r_d},
\qquad
A_0>2R_d,
\qquad
\frac{1}{A_0-2R_d}\le 2R_d,
\qquad
\frac{1}{(A_0-R_d)^2}\le \min\Bigl\{\frac14,\frac{1}{2C_{\diam}^{2}K_{\mathrm{dist}}}\Bigr\}.
\]
If $z\in I_d$ then $\abs{z}\le R_d$, hence for $\abs{\alpha}\ge A_0$,
\[
\abs{h_\alpha(z)}=\frac{1}{\abs{z+\alpha}}
\le \frac{1}{\abs{\alpha}-R_d}
\le \frac{1}{A_0-R_d}
< r_d,
\]
which proves (i). Then (ii) follows immediately from $B(0,r_d)\subset I_d^\circ$. If $w\in h_\alpha(I_d^\circ)\cap h_\beta(I_d^\circ)$, then
$1/w\in(\alpha+I_d^\circ)\cap(\beta+I_d^\circ)$.
Since the selection map $[\cdot]$ is unique on each translate $\alpha+I_d^\circ$,
the translates $\alpha+I_d^\circ$ are pairwise disjoint, and so $\alpha=\beta$.
Thus (iii) holds.
For (iv), since $\abs{\alpha}\ge A_0>2R_d$, the translate $\alpha+\widetilde\Delta_d$ avoids $0$ and $h_\alpha$ is holomorphic on $\widetilde\Delta_d$
with inverse $h_\alpha^{-1}(w)=1/w-\alpha$.
Moreover for $z\in\widetilde\Delta_d$,
\[
\abs{h_\alpha(z)}\le \frac{1}{\abs{\alpha}-2R_d}\le \frac{1}{A_0-2R_d}\le 2R_d,
\]
so $h_\alpha(\widetilde\Delta_d)\subset \widetilde\Delta_d$.
Finally, for $z\in I_d$,
\[
\abs{D h_\alpha(z)}=\frac{1}{\abs{z+\alpha}^{2}}
\le \frac{1}{(\abs{\alpha}-R_d)^2}
\le \frac{1}{(A_0-R_d)^2}
=: \gamma.
\]
This proves (v). The final IFS statement follows from (ii)--(v).
\end{proof}

\begin{lm}\label{lm:2decay}
There exist constants $0<C_1\le C_2<\infty$ and $C_{\diam}\ge 1$ such that:
\begin{enumerate}[label=(\roman*), leftmargin=3.0em]
\item for every $\alpha\in\Ocal_d^{\ge A_0}$ and every $z\in I_d$,
\[
\frac{C_1}{\abs{\alpha}^{2}}\ \le\ \abs{D h_\alpha(z)}\ \le\ \frac{C_2}{\abs{\alpha}^{2}};
\]
\item for every $n\ge 1$ and every word $\boldsymbol{\alpha}=(\alpha_1,\dots,\alpha_n)\in (\Ocal_d^{\ge A_0})^n$,
\[
C_{\diam}^{-1}\prod_{k=1}^n (\abs{\alpha_k}+R_d)^{-2}
\ \le\
\diam\!\bigl(h_{\boldsymbol{\alpha}}(\Delta_d)\bigr)
\ \le\
C_{\diam}\prod_{k=1}^n (\abs{\alpha_k}-R_d)^{-2}.
\]
\end{enumerate}
\end{lm}
\begin{proof}
For (i), if $z\in I_d$ then $\abs{z}\le R_d$ by Lemma~\ref{lm:Id_ball_ext}. Hence
\[
\frac{1}{(\abs{\alpha}+R_d)^2}\ \le\ \frac{1}{\abs{z+\alpha}^2}\ \le\ \frac{1}{(\abs{\alpha}-R_d)^2}.
\]
Since $\abs{\alpha}\ge A_0\ge 1$, we have $\abs{\alpha}+R_d\le (1+R_d)\abs{\alpha}$, which gives the lower bound with $C_1:=(1+R_d)^{-2}$.
Also $\abs{\alpha}-R_d\ge (1-R_d/A_0)\abs{\alpha}$, which gives the upper bound with $C_2:=(1-R_d/A_0)^{-2}$.
For (ii), since $B(0,r_d)\subset \Delta_d^\circ$ by Lemma~\ref{lm:geometry_fd}, Lemma~\ref{lm:NT25_geom}\textup{\ref{it:diam_deriv}} (with $\zeta=0$ and $\delta=r_d$) yields
\[
\diam\!\bigl(h_{\boldsymbol{\alpha}}(\Delta_d)\bigr)\asymp \abs{D h_{\boldsymbol{\alpha}}(0)}
\]
uniformly in $n$ and $\boldsymbol{\alpha}\in (\Ocal_d^{\ge A_0})^n$.
By the chain rule and the bound $\abs{z}\le R_d$ for all $z\in I_d$,
\[
\prod_{k=1}^n (\abs{\alpha_k}+R_d)^{-2}
\ \le\
\abs{D h_{\boldsymbol{\alpha}}(0)}
\ \le\
\prod_{k=1}^n (\abs{\alpha_k}-R_d)^{-2},
\]
since each intermediate point $h_{\alpha_{k+1}}\circ\cdots\circ h_{\alpha_n}(0)$ lies in $I_d$ by Lemma~\ref{lm:trapping}(i).
Combining these bounds gives the claim.
\end{proof}

\begin{lm}\label{lm:fullshift}
Let $\pi:(\Ocal_d^{\ge A_0})^\N\to I_d$ be the coding map for $\Phi_d^{\ge A_0}$ (based at any $\zeta\in I_d$).
Then $\pi((\Ocal_d^{\ge A_0})^\N)\subset B(0,r_d)\subset I_d^\circ$ and $\pi((\Ocal_d^{\ge A_0})^\N)\cap\mathcal N_d=\varnothing$, and for every $\omega=(\alpha_1,\alpha_2,\dots)\in (\Ocal_d^{\ge A_0})^\N,$
the point $x=\pi(\omega)$ satisfies
\[
a_n(x)=\alpha_n\qquad(n\ge 1).
\]
Conversely, if $x\in I_d\setminus\mathcal N_d$ satisfies $a_n(x)\in\Ocal_d^{\ge A_0}$ for all $n\ge 1$, then $x=\pi(\omega)$ for
$\omega=(a_1(x),a_2(x),\dots)\in (\Ocal_d^{\ge A_0})^\N$.
Moreover, the set of points of $\pi((\Ocal_d^{\ge A_0})^\N)$ having more than one address under $\pi$ is countable.
\end{lm}
\begin{proof}
By Lemma~\ref{lm:trapping}(i), each $h_\alpha(I_d)$ is contained in $B(0,r_d)\Subset I_d^\circ$.
Thus every nested cylinder image $h_{\alpha_1}\circ\cdots\circ h_{\alpha_n}(I_d)$ is compactly contained in $I_d^\circ$.
Their intersection is a singleton $x\in I_d^\circ$ by uniform contraction.
We claim $x\notin\mathcal N_d$ and, moreover, $T_d^{n-1}(x)\in h_{\alpha_n}(I_d^\circ)$ for all $n\ge 1$.
Indeed, since $h_{\alpha_2}(I_d)\subset I_d^\circ$, we have
\[
x\in h_{\alpha_1}\circ h_{\alpha_2}(I_d)\subset h_{\alpha_1}(I_d^\circ),
\]
hence $1/x\in \alpha_1+I_d^\circ\subset \C\setminus\mathcal{B}_d$ and so $x\notin B_d$.
Therefore $T_d(x)=h_{\alpha_1}^{-1}(x)$ is well-defined and belongs to
\[
T_d(x)\in h_{\alpha_2}\circ h_{\alpha_3}(I_d)\subset h_{\alpha_2}(I_d^\circ).
\]
Iterating this argument gives $T_d^{n-1}(x)\in h_{\alpha_n}(I_d^\circ)\subset I_d^\circ$ for all $n\ge 1$, hence $T_d^{n-1}(x)\notin B_d$ for all $n$.
Since $h_\alpha$ never vanishes, $x\neq 0$, and thus $x\notin\mathcal N_d$.
For such $x$, the inclusion $x\in h_{\alpha_1}(I_d^\circ)$ implies $1/x\in \alpha_1+I_d^\circ$, hence $[1/x]=\alpha_1$ and $a_1(x)=\alpha_1$.
Applying the same argument to $T_d^{n-1}(x)\in h_{\alpha_n}(I_d^\circ)$ yields $a_n(x)=\alpha_n$ for all $n$.
Conversely, let $x\in I_d\setminus\mathcal N_d$ satisfy $a_n(x)\in\Ocal_d^{\ge A_0}$ for all $n\ge 1$ and set $\omega:=(a_1(x),a_2(x),\dots)\in (\Ocal_d^{\ge A_0})^\N$.
For each $n\ge 1$ we have
\[
x=h_{a_1(x)}\circ h_{a_2(x)}\circ\cdots\circ h_{a_n(x)}\bigl(T_d^n(x)\bigr)\in h_{\omega_1}\circ\cdots\circ h_{\omega_n}(I_d),
\]
since $T_d^n(x)\in I_d$.
Therefore $x\in \bigcap_{n\ge 1} h_{\omega_1}\circ\cdots\circ h_{\omega_n}(I_d)$, which equals $\{\pi(\omega)\}$ by uniform contraction,
hence $x=\pi(\omega)$. Finally, since $\Phi_d^{\ge A_0}$ satisfies \textup{(A4)} by Lemma~\ref{lm:trapping}, Lemma~\ref{lm:coding_countable} implies that $\pi$ fails to be injective only on a countable subset of $\pi((\Ocal_d^{\ge A_0})^\N)$.
\end{proof}

\begin{lm}\label{lm:bounded_overlap}
For $n\ge1$ and $\boldsymbol{\alpha}\in (\Ocal_d^{\ge A_0})^n$ write $\Delta(\boldsymbol{\alpha}):=h_{\boldsymbol{\alpha}}(\Delta_d)$.
There exist constants $c>0$ and $Q\in\N$ such that:
\begin{enumerate}[label=(\roman*), leftmargin=3.0em]
\item for every $n$ and every $\boldsymbol{\alpha}\in (\Ocal_d^{\ge A_0})^n$, the set $\Delta(\boldsymbol{\alpha})$ contains a Euclidean disk of radius $c\,\diam(\Delta(\boldsymbol{\alpha}))$;
\item for every $x\in\C$ and $r>0$, the ball $B(x,r)$ meets at most $Q$ sets $\Delta(\boldsymbol{\alpha})$
(over all $n\ge1$ and $\boldsymbol{\alpha}\in (\Ocal_d^{\ge A_0})^n$) satisfying
\[
r \le \diam(\Delta(\boldsymbol{\alpha})) < 2r.
\]
\end{enumerate}
\end{lm}
\begin{proof}
Let $r_d>0$ be as in Lemma~\ref{lm:geometry_fd}, so $B(0,r_d)\subset \Delta_d^\circ$.
Let $K_{\mathrm{dist}}\ge 1$ and $C_{\diam}\ge 1$ be as in Lemma~\ref{lm:trapping}.
Applying Lemma~\ref{lm:NT25_geom}\textup{\ref{it:inscribed_disk}} with $\zeta=0$ and $\delta=r_d$, for every $n\ge 1$ and $\boldsymbol{\alpha}\in (\Ocal_d^{\ge A_0})^n$
the set $\Delta(\boldsymbol{\alpha})$ contains a Euclidean disk of radius at least
\[
\frac{r_d}{3K_{\mathrm{dist}}}\,\abs{D h_{\boldsymbol{\alpha}}(0)}.
\]
By Lemma~\ref{lm:NT25_geom}\textup{\ref{it:diam_deriv}} (with $\zeta=0$ and $\delta=r_d$), we have $\abs{D h_{\boldsymbol{\alpha}}(0)}\asymp \diam(\Delta(\boldsymbol{\alpha}))$ uniformly in $n,\boldsymbol{\alpha}$,
hence (i). For (ii), fix $x\in\C$ and $r>0$ and let $\mathscr{F}$ be the collection of all $\Delta(\boldsymbol{\alpha})$ with
\[
r\le \diam(\Delta(\boldsymbol{\alpha}))<2r
\]
that meet $B(x,r)$.
For each $\Delta(\boldsymbol{\alpha})\in\mathscr{F}$ choose an inscribed disk $D(\boldsymbol{\alpha})\subset \Delta(\boldsymbol{\alpha})$ from (i).
Then $D(\boldsymbol{\alpha})\subset B(x,3r)$ and $\mathrm{rad}(D(\boldsymbol{\alpha}))\ge c\,r$. We claim that the disks $D(\boldsymbol{\alpha})$ are pairwise disjoint.
Assume $\Delta(\boldsymbol{\alpha})^\circ\cap \Delta(\boldsymbol{\beta})^\circ\neq\varnothing$.
Let $\boldsymbol{\nu}$ be the maximal common prefix of $\boldsymbol{\alpha},\boldsymbol{\beta}$.
If after removing $\boldsymbol{\nu}$ the next digits are distinct, say
$\boldsymbol{\alpha}=(\boldsymbol{\nu},\eta,\dots)$ and $\boldsymbol{\beta}=(\boldsymbol{\nu},\eta',\dots)$ with $\eta\neq\eta'$, then applying
$h_{\boldsymbol{\nu}}^{-1}$ to a point of intersection yields a point in
$h_{\eta}(I_d^\circ)\cap h_{\eta'}(I_d^\circ)$, contradicting Lemma~\ref{lm:trapping}(iii).
Thus one of the two words is a prefix of the other.
Assume, for contradiction, that $\boldsymbol{\alpha}$ is a strict prefix of $\boldsymbol{\beta}$, so that
$\boldsymbol{\beta}=(\boldsymbol{\alpha},\eta,\dots)$ for some $\eta\in\Ocal_d^{\ge A_0}$.
Then $\Delta(\boldsymbol{\beta})\subset \Delta(\boldsymbol{\alpha}\eta)$, and Lemma~\ref{lm:NT25_geom}\textup{\ref{it:diam_deriv}} gives
\begin{equation}
\diam\bigl(\Delta(\boldsymbol{\alpha}\eta)\bigr)
\le
C_{\diam}\,\abs{D h_{\boldsymbol{\alpha}\eta}(0)}
=
C_{\diam}\,\abs{D h_{\boldsymbol{\alpha}}(h_{\eta}(0))}\,\abs{D h_{\eta}(0)}.
    \label{eqn:a}
\end{equation}
By bounded distortion for $h_{\boldsymbol{\alpha}}$ (Lemma~\ref{lm:NT25_geom}\textup{\ref{it:koebe_dist}}),
\(
\abs{D h_{\boldsymbol{\alpha}}(h_{\eta}(0))}
\le
K_{\mathrm{dist}}\,\abs{D h_{\boldsymbol{\alpha}}(0)}.
\) Using Lemma~\ref{lm:NT25_geom}\textup{\ref{it:diam_deriv}}, now for $\Delta(\boldsymbol{\alpha})$, we obtain
\(
\abs{D h_{\boldsymbol{\alpha}}(0)}
\le
C_{\diam}\,\diam\bigl(\Delta(\boldsymbol{\alpha})\bigr).
\) Therefore
\[
\diam\bigl(\Delta(\boldsymbol{\beta})\bigr)
\le
\diam\bigl(\Delta(\boldsymbol{\alpha}\eta)\bigr)
\le
C_{\diam}^{2}K_{\mathrm{dist}}\,\abs{D h_{\eta}(0)}\,\diam\bigl(\Delta(\boldsymbol{\alpha})\bigr).
\]
by \eqref{eqn:a}. By Lemma~\ref{lm:trapping}(v), $\abs{D h_{\eta}(0)}\le \gamma$ and
\(
C_{\diam}^{2}K_{\mathrm{dist}}\gamma\le 1/2,
\)
and so
\[
\diam\bigl(\Delta(\boldsymbol{\beta})\bigr)
\le
\tfrac12\diam\bigl(\Delta(\boldsymbol{\alpha})\bigr)
< r,
\]
since $\diam(\Delta(\boldsymbol{\alpha}))<2r$ for every member of $\mathscr F$. This contradicts $r\le \diam(\Delta(\boldsymbol{\beta}))$.
Therefore strict prefix cannot occur, so the words are incomparable and the interiors are disjoint. Hence the disks are disjoint. Each disk $D(\boldsymbol{\alpha})$ has radius at least $c r$, hence area at least $\pi(c r)^2$.
Since all $D(\boldsymbol{\alpha})$ lie in $B(x,3r)$, a packing bound yields
\[
\#\mathscr{F}\le \frac{\mathrm{area}(B(x,3r))}{\pi(c r)^2}=\frac{9}{c^2}.
\]
Taking $Q:=\lceil 9/c^2\rceil$ gives (ii).
\end{proof}

\begin{lm}\label{lm:prefix_conformal}
For every admissible finite word $\mathbf{b}=(b_1,\dots,b_k)\in\Ocal_d^k$,
the map $h_{\mathbf{b}}:=h_{b_1}\circ\cdots\circ h_{b_k}$ is a M\"obius transformation which is holomorphic and injective
on an open neighborhood of $I_d$.
In particular, $h_{\mathbf{b}}$ is bi-Lipschitz on $I_d$ and preserves Hausdorff dimension on sets
$E\subset I_d$:
\[
\dim_H(h_{\mathbf{b}}(E))=\dim_H(E).
\]
\end{lm}

\begin{proof}
By Lemma~\ref{lm:prefix_mobius_formula},
\[
h_{\mathbf b}(z)=\frac{P_{k-1}z+P_k}{Q_{k-1}z+Q_k},
\qquad
z_*:=-\frac{Q_k}{Q_{k-1}},
\]
and $z_*$ is the unique pole of $h_{\mathbf b}$. If $k=1$, then $z_*=-b_1$, so $\abs{z_*}=\abs{b_1}\ge 1>R_d$.
Assume $k\ge 2$.
Choose $x\in I_d\setminus\mathcal N_d$ with $(a_1(x),\dots,a_k(x))=\mathbf b$.
Then $(Q_{k-1},Q_k)$ is the denominator pair of the convergents of $x$, and Lemma~\ref{lm:denom_growth} gives
\[
\abs{Q_{k-1}}<\abs{Q_k}.
\]
Hence
\[
\abs{z_*}=\frac{\abs{Q_k}}{\abs{Q_{k-1}}}>1>R_d.
\]
By Lemma~\ref{lm:Id_ball_ext}, we have $I_d\subset B(0,R_d)$, so the pole lies outside the open ball $B(0,R_d)$.
Therefore $h_{\mathbf b}$ is holomorphic on $B(0,R_d)$, which is an open neighborhood of $I_d$.
Since a nonconstant M\"obius transformation is injective on its domain of holomorphy,
$h_{\mathbf b}$ is injective on a neighborhood of $I_d$. Define
\[
\Psi(z,w):=
\begin{cases}
\dfrac{\abs{h_{\mathbf b}(z)-h_{\mathbf b}(w)}}{\abs{z-w}}, & z\neq w,\\[2ex]
\abs{D h_{\mathbf b}(z)}, & z=w.
\end{cases}
\]
Since $h_{\mathbf b}$ is $C^1$ on a neighborhood of $I_d$, the function $\Psi$ is continuous on $I_d\times I_d$.
Because $h_{\mathbf b}$ is injective on a neighborhood of $I_d$, we have $\Psi(z,w)>0$ for all $(z,w)\in I_d\times I_d$.
Compactness therefore gives constants $0<m_{\mathbf b}\le M_{\mathbf b}<\infty$ such that
\[
m_{\mathbf b}\le \Psi(z,w)\le M_{\mathbf b}
\qquad((z,w)\in I_d\times I_d).
\]
Equivalently,
\[
m_{\mathbf b}\abs{z-w}\le \abs{h_{\mathbf b}(z)-h_{\mathbf b}(w)}\le M_{\mathbf b}\abs{z-w}
\qquad(z,w\in I_d),
\]
so $h_{\mathbf b}$ is bi-Lipschitz on $I_d$.
Bi-Lipschitz invariance gives
\[
\dim_H(h_{\mathbf b}(E))=\dim_H(E)
\qquad(E\subset I_d).
\]
\end{proof}

\begin{lm}\label{lm:local_prefix}
Let $x\in I_d\setminus\mathcal N_d$ and $n\ge 1$.
Write
\(
\mathbf b:=(a_1(x),\dots,a_n(x))\) and $y:=T_d^n(x)\in I_d^\circ$. Then there exists an open neighborhood $D\subset I_d^\circ$ of $y$ such that for every $z\in D$, we have
\(
T_d^n\bigl(h_{\mathbf b}(z)\bigr)=z
\)
and
\(
a_j\bigl(h_{\mathbf b}(z)\bigr)=a_j(x)
\) for $1\le j\le n$. In particular,
\(
h_{\mathbf b}(D)\subset \ip{\mathbf b}.
\)
\end{lm}

\begin{proof}
For $2\le j\le n$, set
\[
\psi_j:=h_{a_j(x)}\circ\cdots\circ h_{a_n(x)}.
\]
Since $x\notin\mathcal N_d$, we have $T_d^m(x)\notin B_d\cup\{0\}$ for $0\le m\le n-1$, and hence
\[
T_d^{m+1}(x)=\frac{1}{T_d^m(x)}-a_{m+1}(x)\in I_d^\circ
\qquad(0\le m\le n-1).
\]
Therefore
\[
\psi_j(y)=T_d^{j-1}(x)\in I_d^\circ
\qquad(2\le j\le n).
\]
Each $\psi_j$ is holomorphic on a neighborhood of $I_d$ by Lemma~\ref{lm:prefix_conformal}.
Since $I_d^\circ$ is open and the family $\{\psi_j:2\le j\le n\}$ is finite, there exists an open neighborhood $D\subset I_d^\circ$ of $y$ such that
\[
\psi_j(D)\subset I_d^\circ
\qquad(2\le j\le n).
\]
Fix $z\in D$ and define
\[
x_{n+1}:=z,
\qquad
x_j:=h_{a_j(x)}(x_{j+1})
\qquad(1\le j\le n).
\]
Then $x_j=\psi_j(z)\in I_d^\circ$ for $2\le j\le n$, while $x_{n+1}=z\in I_d^\circ$.
Hence, for $1\le j\le n$,
\[
\frac{1}{x_j}=a_j(x)+x_{j+1}\in a_j(x)+I_d^\circ.
\]
It follows that $x_j\in I_d\setminus(B_d\cup\{0\})$, that $a_1(x_j)=a_j(x)$, and that $T_d(x_j)=x_{j+1}$ for every $1\le j\le n$.
Since $x_1=h_{\mathbf b}(z)$, repeated application of $T_d$ gives
\[
T_d^n\bigl(h_{\mathbf b}(z)\bigr)=z
\]
and
\[
a_j\bigl(h_{\mathbf b}(z)\bigr)=a_j(x)
\qquad(1\le j\le n).
\]
Thus $h_{\mathbf b}(z)\in \ip{\mathbf b}$.
\end{proof}

\begin{lm}\label{lm:cofinite_tail}
Let $S\subset\Ocal_d$ be infinite set and denote $S':=S\cap\Ocal_d^{\ge A_0}$. Define
\[
E(S'):=\set{\pi(\omega):\ \omega\in (S')^\N,\ \abs{\omega_n}\to\infty}.
\]
Then $E(S')\subset F_d(S)$.
Moreover,
\[
F_d(S)\subset \bigcup_{\mathbf{b}} h_{\mathbf{b}}\bigl(E(S')\bigr),
\]
where the union ranges over all possibly empty admissible finite words $\mathbf{b}$.
\end{lm}

\begin{proof}
If $x\in E(S')$, then $x=\pi(\omega)$ for some $\omega\in(S')^\N$ with $\abs{\omega_n}\to\infty$.
By Lemma~\ref{lm:fullshift}, $a_n(x)=\omega_n\in S'\subset S$ for all $n$ and $\abs{a_n(x)}\to\infty$, hence $x\in F_d(S)$. Conversely, if $x\in F_d(S)$ then $\abs{a_n(x)}\to\infty$, so $a_n(x)\in S'$ for all $n\ge N$ for some $N\ge 2$. Set $\mathbf{b}:=(a_1(x),\dots,a_{N-1}(x))$ (admissible) and $y:=T_d^{N-1}(x)$.
Then $a_n(y)=a_{N-1+n}(x)\in S'$ for all $n\ge 1$ and $\abs{a_n(y)}\to\infty$, so $y\in E(S')$ by Lemma~\ref{lm:fullshift}.
Since $x=h_{\mathbf{b}}(y)$, this gives the desired inclusion.
\end{proof}

\begin{proof}[Proof of Theorem~\ref{thm:HirstMain}]
We first pass to the cofinite large-digit subsystem $S':=S\cap\Ocal_d^{\ge A_0}$, identify its tail set with the $2$-decaying IFS $\Phi_d^{\ge A_0}$, and then transfer the dimension statement back to $F_d(S)$ and $F_d(S,f)$ through the finite-prefix decomposition from Lemma~\ref{lm:cofinite_tail}.
\noindent\textit{(i) Unrestricted case.} Fix $d\in\{1,2,3,7,11\}$. Let $A_0$ be as in Lemma~\ref{lm:trapping}. By Lemma~\ref{lm:trapping} and~\ref{lm:2decay}, the IFS $\Phi_d^{\ge A_0}=\{h_\alpha\}_{\alpha\in\Ocal_d^{\ge A_0}}$ on $\Delta_d:=I_d$
satisfies \textup{\ref{cond:A1}}--\textup{\ref{cond:A4}} and is $2$-decaying with respect to $\alpha\mapsto\abs{\alpha}$.
Write $L':=L'(\Phi_d^{\ge A_0})$ for the subset of the limit set consisting of points with a unique address.
For an infinite subset $A\subset\Ocal_d^{\ge A_0}$, define
\[
F_{\Phi_d^{\ge A_0}}(A)
:=
\set{x\in L':\ \pi^{-1}(x)=\{\omega(x)\},\ \omega(x)\in A^\N,\ \abs{\omega_n(x)}\to\infty},
\]
and, for a cutoff function $g:\N\to[1,\infty)$,
\[
F_{\Phi_d^{\ge A_0}}(A,g)
:=
\set{x\in F_{\Phi_d^{\ge A_0}}(A):\ \abs{\omega_n(x)}\le g(n)\ \forall n\ge 1}.
\]
Let $S\subset\Ocal_d$ be infinite and set $S':=S\cap\Ocal_d^{\ge A_0}$.
Since $\Ocal_d^{\ge A_0}$ is cofinite in $\Ocal_d$, the set $S\setminus S'$ is finite. By Lemma~\ref{lm:tau_basic}(i), $\tau(S')=\tau(S)$.
Let $E(S')$ be as in Lemma~\ref{lm:cofinite_tail}. Since every point of $E(S')$ has an address in $(S')^\N$, we have
\(
E(S')\setminus F_{\Phi_d^{\ge A_0}}(S')
\subset
L(\Phi_d^{\ge A_0})\setminus L'(\Phi_d^{\ge A_0}).
\) Since $\Phi_d^{\ge A_0}$ satisfies \textup{(A4)}, Lemma~\ref{lm:coding_countable} implies that
\(
E(S')\setminus F_{\Phi_d^{\ge A_0}}(S')
\) is countable. Therefore, by Theorem~\ref{thm:hirst_external}, Lemma~\ref{lm:tau_absS}, and Lemma~\ref{lm:tau_basic}(i), we have
\[
\dim_H E(S')
=\dim_H F_{\Phi_d^{\ge A_0}}(S')
=\frac{\tau(\abs{S'})}{2}
=\frac{\tau(S')}{2}
=\frac{\tau(S)}{2}.
\]
By Lemma~\ref{lm:cofinite_tail}, $E(S')\subset F_d(S)$, hence $\dim_H F_d(S)\ge \tau(S)/2$.
Conversely, Lemma~\ref{lm:cofinite_tail} gives
\(
F_d(S)\subset \bigcup_{\mathbf{b}} h_{\mathbf{b}}\bigl(E(S')\bigr),
\) with $\mathbf{b}$ ranging over possibly empty admissible finite words, and with $h_{\emptyset}=\mathrm{id}$.
Since $E(S')\subset I_d$ and each $h_{\mathbf{b}}$ is bi-Lipschitz on $I_d$ by Lemma~\ref{lm:prefix_conformal}, we obtain
\[
\dim_H h_{\mathbf{b}}\bigl(E(S')\bigr)=\dim_H E(S')=\frac{\tau(S)}{2}.
\]
Taking the countable union yields $\dim_H F_d(S)\le \tau(S)/2$, hence $\dim_H F_d(S)=\tau(S)/2$.
\smallskip
\noindent\textit{(ii) Cutoff case.}
Fix $x_0\in F_d(S,f)$.
Choose $N_0\ge 2$ such that $a_n(x_0)\in S'$ and $f(n)\ge m_{S'}:=\min\{\abs{\alpha}:\alpha\in S'\}$ for all $n\ge N_0$.
Choose $\alpha_*\in S'$ with $\abs{\alpha_*}=m_{S'}$.
Set $\mathbf{b}:=(a_1(x_0),\dots,a_{N_0-1}(x_0))$ and define the shifted cutoff $g(n):=f(n+N_0-1)$.
Let
\[
E(S',g):=\set{\pi(\omega):\ \omega\in (S')^\N,\ \abs{\omega_n}\to\infty,\ \abs{\omega_n}\le g(n)\ \forall n}.
\]
Since $g(n)\ge m_{S'}$ for every $n\ge 1$, we have
\[
\alpha_*\in \set{\alpha\in S':\ \abs{\alpha}\le g(n)}
\qquad(n\ge 1).
\]
Hence Theorem~\ref{thm:hirst_external} applies to $F_{\Phi_d^{\ge A_0}}(S',g)$.
Again Lemma~\ref{lm:coding_countable} shows that
\(
E(S',g)\setminus F_{\Phi_d^{\ge A_0}}(S',g)
\) is countable, and therefore
\begin{equation}
\dim_H E(S',g)
=\dim_H F_{\Phi_d^{\ge A_0}}(S',g)
=\frac{\tau(\abs{S'})}{2}
=\frac{\tau(S)}{2}
    \label{eqn:1234}
\end{equation}
by Lemma~\ref{lm:tau_absS}. Set $y_0:=T_d^{N_0-1}(x_0)$ and write $\omega^0:=(a_1(y_0),a_2(y_0),\dots)\in (S')^\N$.
By Lemma~\ref{lm:fullshift}, $y_0=\pi(\omega^0)\in B(0,r_d)\subset I_d^\circ$.
By Lemma~\ref{lm:local_prefix}, there exists an open neighborhood $D\subset I_d^\circ$ of $y_0$ such that
\[
h_{\mathbf b}(D)\subset \ip{\mathbf b}
\qquad\text{and}\qquad
T_d^{N_0-1}\bigl(h_{\mathbf b}(y)\bigr)=y\qquad (y\in D).
\]
For $\ell\ge 1$ define
\[
Y_\ell:=h_{a_{1}(y_0)}\circ\cdots\circ h_{a_{\ell}(y_0)}(I_d).
\]
Since $y_0=\pi(\omega^0)$, we have $y_0\in Y_\ell$ for every $\ell$.
By Lemma~\ref{lm:trapping}(v), we have
\(
\diam(Y_\ell)\le \gamma^\ell \diam(I_d),
\)
where $\gamma\in(0,1)$ is the uniform contraction constant.
Hence $\diam(Y_\ell)\to 0$, and therefore there exists $\ell\ge 1$ such that $Y_\ell\subset D$.
Since $\ell\ge1$, we have
\[
Y_\ell\subset h_{a_1(y_0)}(I_d)\subset B(0,r_d)\subset I_d
\]
by Lemma~\ref{lm:trapping}(i) and Lemma~\ref{lm:geometry_fd}. Define $g_\ell(n):=g(n+\ell)$ and set
\[
E(S',g_\ell):=\set{\pi(\omega):\ \omega\in (S')^\N,\ \abs{\omega_n}\to\infty,\ \abs{\omega_n}\le g_\ell(n)\ \forall n}.
\]
Since $g_\ell(n)=f(n+N_0+\ell-1)\ge m_{S'}$ for every $n\ge 1$, we have
\[
\dim_H E(S',g_\ell)=\frac{\tau(S)}{2}.
\]
by \eqref{eqn:1234}. Set
\(
K_0:=h_{a_{1}(y_0)}\circ\cdots\circ h_{a_{\ell}(y_0)}\bigl(E(S',g_\ell)\bigr)\subset Y_\ell.
\)
We claim that
\begin{equation}
    K_0\subset E(S',g).
    \label{eqn:claim1}
\end{equation}
Indeed, if $y\in K_0$ then $y=h_{a_{1}(y_0)}\circ\cdots\circ h_{a_{\ell}(y_0)}(z)$
for some $z\in E(S',g_\ell)$, and Lemma~\ref{lm:fullshift} gives
\[
a_j(y)=a_j(y_0)\in S' \qquad(1\le j\le \ell),
\]
while for $n\ge 1$,
\(
a_{\ell+n}(y)=a_n(z)\in S'.
\) Moreover,
\(
\abs{a_j(y_0)}=\abs{a_{N_0-1+j}(x_0)}\le f(N_0-1+j)=g(j)
\) for $1\le j\le \ell$. Since $z\in E(S',g_\ell)$, we see that
\(
\abs{a_{\ell+n}(y)}=\abs{a_n(z)}\le g_\ell(n)=g(\ell+n)
\) for $n\geq 1$. Moreover $\abs{a_n(y)}\to\infty$ because $\abs{a_n(z)}\to\infty$. Hence $y\in E(S',g)$. Next we claim that 
\begin{equation}
h_{\mathbf{b}}(K_0)\subset F_d(S,f).
    \label{eqn:claim2}
\end{equation}
Let $y\in K_0$ and set $x:=h_{\mathbf{b}}(y)$. Since $y\in K_0\subset D$, Lemma~\ref{lm:local_prefix} gives
\[
x\in \ip{\mathbf b}
\qquad\text{and}\qquad
T_d^{N_0-1}(x)=y.
\]
Therefore
\[
a_j(x)=b_j=a_j(x_0)\in S,
\qquad
\abs{a_j(x)}=\abs{b_j}\le f(j)
\qquad(1\le j\le N_0-1),
\]
and for every $n\ge 1$,
\[
a_{N_0-1+n}(x)=a_n(y)\in S'\subset S,
\qquad
\abs{a_{N_0-1+n}(x)}=\abs{a_n(y)}\le g(n)=f(N_0-1+n).
\]
Since $y\in E(S',g)$ by \eqref{eqn:claim1}, we also have $\abs{a_n(y)}\to\infty$, and thus $\abs{a_n(x)}\to\infty$.
Hence $x\in F_d(S,f)$, proving $h_{\mathbf{b}}(K_0)\subset F_d(S,f)$. Since $E(S',g_\ell)\subset I_d$ and the length-$\ell$ prefix map
$h_{a_{1}(y_0)}\circ\cdots\circ h_{a_{\ell}(y_0)}$ is bi-Lipschitz on $I_d$ (Lemma~\ref{lm:prefix_conformal}),
\[
\dim_H K_0=\dim_H E(S',g_\ell)=\frac{\tau(S)}{2}.
\]
Since $K_0\subset Y_\ell\subset I_d$, the map $h_{\mathbf{b}}$ is also bi-Lipschitz on $K_0$ by Lemma~\ref{lm:prefix_conformal}. Hence
\[
\dim_H F_d(S,f)\ \ge\ \dim_H h_{\mathbf{b}}(K_0)=\dim_H K_0=\frac{\tau(S)}{2}.
\]
The opposite inequality follows from $F_d(S,f)\subset F_d(S)$.
\end{proof}

\begin{proof}[Proof of Corollary~\ref{cor:tau_examples_intro}]
(i) By Lemma~\ref{lm:lattice_count}(ii),
\[
\sum_{\alpha\in\Ocal_d\setminus\{0\}} \abs{\alpha}^{-t}
\asymp
\sum_{k\ge0} \#\set{\alpha:\ 2^k\le\abs{\alpha}<2^{k+1}}\,2^{-kt}
\asymp
\sum_{k\ge0} 2^{k(2-t)}.
\]
Hence the series converges if and only if $t>2$, so $\tau(\Ocal_d)=2$.
The dimension statement follows from Theorem~\ref{thm:HirstMain}.
(ii) Since $\sum_{n\in\Z\setminus\{0\}} |n|^{-t}$ converges if and only if $t>1$, we have $\tau(\Z)=1$.
Apply Theorem~\ref{thm:HirstMain}.
(iii) We have $\sum_{n\ge1} |n^p|^{-t}=\sum_{n\ge1} n^{-pt}$, which converges if and only if $pt>1$, hence $\tau(S)=1/p$.
Apply Theorem~\ref{thm:HirstMain}.
\end{proof}


\section{Sparse patterns}

Fix a cutoff function $f:\N\to[1,\infty)$ with $f(n)\to\infty$.
Fix $m\ge 1$. Let $(N_r)_{r\ge 1}$ be a strictly increasing sequence of integers and let $\mathbf{b}^{(r)}=(b^{(r)}_1,\dots,b^{(r)}_m)\in(\Ocal_d)^m$ be blocks.
Assume
\[
\min_{1\le j\le m}\abs{b^{(r)}_j}\to\infty
\qquad(r\to\infty),
\]
and the cutoff bounds
\[
\abs{b^{(r)}_j}\le f(N_r+j-1)
\qquad(r\ge 1,\ 1\le j\le m).
\]
Define
\[
F_d(S,f;\mathbf{b}^{(\cdot)},N_\cdot):=
\set{x\in F_d(S,f):\ (a_{N_r}(x),\dots,a_{N_r+m-1}(x))=\mathbf{b}^{(r)}\ \text{for infinitely many }r}.
\]

We prove Theorem~\ref{thm:pattern_intro} by constructing a sparse non-autonomous subsystem.
We choose a reference point in the target set, select the variable levels from finite subsets with uniformly positive $s$-mass, and insert the prescribed blocks along a sparse subsequence whose cumulative contribution to the pressure is negligible.

\begin{lm}\label{lm:tail_select_cutoff}
Let $S\subset\Ocal_d$ be infinite.
Fix $t$ with $0<t<\tau(S)$ and choose $\varepsilon>0$ so that $s:=t+\varepsilon<\tau(S)$.
Let $f:\N\to[1,\infty)$ satisfy $f(n)\to\infty$.
Then there exist $n_0\ge 2$, finite sets $E_n\subset S$ for $n\ge n_0$, and a sequence $R_n\to\infty$ such that for all $n\ge n_0$,
\[
R_n\le \abs{\alpha}\le f(n)\qquad(\alpha\in E_n),
\qquad
1\le \sum_{\alpha\in E_n}(\abs{\alpha}+R_d)^{-s}\le 2,
\]
and moreover
\[
\lim_{n\to\infty}\frac{\log \# E_n}{n}=0.
\]
\end{lm}
\begin{proof}
Let $A_0$ be as in Lemma~\ref{lm:trapping} and set $S':=S\cap \Ocal_d^{\ge A_0}$.
Since $\Ocal_d\setminus \Ocal_d^{\ge A_0}$ is finite, Lemma~\ref{lm:tau_basic}(i) gives $\tau(S')=\tau(S)$.
By Lemma~\ref{lm:tau_basic}(ii) with $c=R_d$,
\[
\sum_{\alpha\in S'_*}(\abs{\alpha}+R_d)^{-s}=\infty.
\]
Enumerate $S'_* = \{\xi_1,\xi_2,\dots\}$ so that $\abs{\xi_1}\le \abs{\xi_2}\le \cdots$.
Since $\abs{\xi_n}\to\infty$, the terms $(\abs{\xi_n}+R_d)^{-s}$ tend to $0$.
Hence we may choose inductively disjoint finite consecutive blocks
\[
F_k\subset S'_*,\qquad k\ge 1,
\]
such that
\[
1\le \sum_{\alpha\in F_k}(\abs{\alpha}+R_d)^{-s}\le 2
\qquad(k\ge 1),
\]
and
\[
\min_{\alpha\in F_k}\abs{\alpha}\to\infty
\qquad(k\to\infty).
\]
For each $k$ set
\[
m_k:=\min_{\alpha\in F_k}\abs{\alpha},
\qquad
M_k:=\max_{\alpha\in F_k}\abs{\alpha}.
\]
Since $f(n)\to\infty$, for each $k$ there exists $N_k$ such that
\[
f(n)\ge M_k\qquad(n\ge N_k).
\]
By enlarging $N_k$ if necessary, we may assume that $(N_k)_{k\ge 1}$ is strictly increasing and
\[
\log \# F_k\le \frac{N_k}{k}\qquad(k\ge 1).
\]
Define $n_0:=\max\{N_1,2\}$ and, for $n\ge n_0$,
\[
E_n:=F_k,
\qquad
R_n:=m_k
\qquad\text{whenever }N_k\le n < N_{k+1}.
\]
Then $E_n\subset S'\subset S$,
\[
R_n\le \abs{\alpha}\le f(n)\qquad(\alpha\in E_n),
\qquad
1\le \sum_{\alpha\in E_n}(\abs{\alpha}+R_d)^{-s}\le 2,
\]
and $R_n\to\infty$ because $m_k\to\infty$.
Moreover, if $N_k\le n < N_{k+1}$, then
\[
\frac{\log \# E_n}{n}=\frac{\log \# F_k}{n}
\le \frac{\log \# F_k}{N_k}
\le \frac{1}{k},
\]
which tends to $0$ as $n\to\infty$.
\end{proof}

\begin{lm}\label{lm:cyl_shrink}
Let $A_0$ be as in Lemma~\ref{lm:trapping} and set $\mathcal I:=\Ocal_d^{\ge A_0}$.
Let $\omega=(\omega_1,\omega_2,\dots)\in \mathcal I^\N$ and put $y:=\pi(\omega)\in I_d$.
If $D\subset I_d$ is an open neighborhood of $y$, then there exists $\ell\ge 1$ such that
\[
h_{\omega_1}\circ\cdots\circ h_{\omega_\ell}(I_d)\subset D.
\]
\end{lm}
\begin{proof}
Choose $\rho>0$ such that $B(y,\rho)\subset D$.
Since $y\in h_{\omega_1}\circ\cdots\circ h_{\omega_\ell}(I_d)$ for every $\ell\ge1$ and
\[
\diam\bigl(h_{\omega_1}\circ\cdots\circ h_{\omega_\ell}(I_d)\bigr)\to 0
\qquad(\ell\to\infty)
\]
by Lemma~\ref{lm:2decay}(ii), there exists $\ell$ such that
\[
\diam\bigl(h_{\omega_1}\circ\cdots\circ h_{\omega_\ell}(I_d)\bigr)<\rho.
\]
Hence
\[
h_{\omega_1}\circ\cdots\circ h_{\omega_\ell}(I_d)\subset B(y,\rho)\subset D.
\]
\end{proof}

\begin{proof}[Proof of Theorem~\ref{thm:pattern_intro}]
Set
\[
E:=F_d(S,f;\mathbf{b}^{(\cdot)},N_\cdot)\subset F_d(S,f).
\]
If $E=\varnothing$, there is nothing to prove.
By Theorem~\ref{thm:HirstMain},
\[
\dim_H(E)\le \dim_H(F_d(S,f))=\frac{\tau(S)}{2}.
\]
If $\tau(S)=0$, this already gives $\dim_H(E)=0$.
Thus we may assume $\tau(S)>0$.
Fix $t$ with $0<t<\tau(S)$ and choose $\varepsilon>0$ so that
\[
s:=t+\varepsilon<\tau(S),
\qquad
c_0:=\frac{\varepsilon\log 2}{8}>0.
\]
Let $A_0$ be as in Lemma~\ref{lm:trapping} and set $S':=S\cap \Ocal_d^{\ge A_0}$.
By Lemma~\ref{lm:tau_basic}(i), $\tau(S')=\tau(S)$.
Apply Lemma~\ref{lm:tail_select_cutoff} to $S'$ and $f$.
Choose $x_*\in E$.
After enlarging $n_0$ if necessary, we may still assume
\[
a_n(x_*)\in S',
\qquad
R_n\ge 2
\qquad(n\ge n_0).
\]
Since $b_j^{(r)}\in S$ for all $r\ge 1$ and $1\le j\le m$, while
\[
\min_{1\le j\le m}\abs{b_j^{(r)}}\to\infty,
\]
the set
\[
\mathcal R:=\set{r\ge 1:\ b_j^{(r)}\in S' \text{ for all }1\le j\le m}
\]
is cofinite.
\smallskip\noindent\emph{Reference orbit.}
Set
\[
y_*:=T_d^{n_0-1}(x_*),
\qquad
\mathbf c:=(a_1(x_*),\dots,a_{n_0-1}(x_*)).
\]
Since
\[
a_1(y_*)=a_{n_0}(x_*)\in S'\subset \Ocal_d^{\ge A_0},
\]
Lemma~\ref{lm:trapping}(i) gives $y_*\in B(0,r_d)\subset I_d^\circ$.
Lemma~\ref{lm:local_prefix} applied to $x_*$ and $n_0-1$ yields an open neighborhood $D\subset I_d^\circ$ of $y_*$ such that
\[
h_{\mathbf c}(D)\subset \ip{\mathbf c}
\qquad\text{and}\qquad
T_d^{n_0-1}\bigl(h_{\mathbf c}(y)\bigr)=y
\quad (y\in D).
\]
Choose an open set $U_*$ such that
\[
y_*\in U_*\Subset D\cap B(0,r_d).
\]
Write
\[
\omega^*:=(a_1(y_*),a_2(y_*),\dots)\in (S')^\N\subset (\Ocal_d^{\ge A_0})^\N.
\]
By Lemma~\ref{lm:fullshift}, $y_*=\pi(\omega^*)$.
Lemma~\ref{lm:cyl_shrink} yields $\ell\ge 1$ such that
\begin{equation}\label{eq:anchor-window}
Y_\ell:=h_{a_1(y_*)}\circ\cdots\circ h_{a_\ell(y_*)}(I_d)\subset U_*\Subset D\cap B(0,r_d).
\end{equation}
\smallskip\noindent\emph{Sparse subsystem.}
For $r\in\mathcal R$, set
\[
\widetilde B_r:=\prod_{j=1}^m (\abs{b_j^{(r)}}+R_d).
\]
Then
\[
0\le \log \widetilde B_r-\log\Bigl(\prod_{j=1}^m\abs{b_j^{(r)}}\Bigr)
=\sum_{j=1}^m \log\Bigl(1+\frac{R_d}{\abs{b_j^{(r)}}}\Bigr)\to 0,
\]
so the hypothesis on the prescribed blocks gives
\[
\liminf_{r\to\infty}\frac{\log \widetilde B_r}{N_r}=0.
\]
Since $\mathcal R$ is cofinite, $N_r\to\infty$, and $\liminf_{r\to\infty}(\log \widetilde B_r)/N_r=0$, we may choose recursively a strictly increasing sequence $(r_k)_{k\ge 1}$ in $\mathcal R$ such that
\begin{align}
N_{r_k}&\ge n_0+\ell, \label{eq:rk-basic}\\
N_{r_k}&\ge N_{r_{k-1}}+m \qquad (k\ge 2), \label{eq:rk-disjoint}\\
t\log \widetilde B_{r_k}&\le c_0 N_{r_k}, \label{eq:rk-smallratio}\\
N_{r_k}&\ge c_0^{-1}\Bigl(t\sum_{i=1}^{k-1}\log \widetilde B_{r_i}+8c_0mk+1\Bigr). \label{eq:rk-dominate}
\end{align}
This is possible because the first three conditions involve only finitely many previously chosen terms, while the last one asks that $N_{r_k}$ dominate a fixed finite quantity.
From \eqref{eq:rk-smallratio} and \eqref{eq:rk-dominate},
\begin{equation}\label{eq:cumulative-loss-bound}
t\sum_{i=1}^{k}\log \widetilde B_{r_i}+8c_0mk
\le 2c_0 N_{r_k}
\qquad(k\ge 1).
\end{equation}
Set
\[
\widehat N_k:=N_{r_k}-(n_0-1).
\]
By \eqref{eq:rk-basic} and \eqref{eq:rk-disjoint},
\[
\widehat N_1\ge \ell+1,
\qquad
\widehat N_{k+1}\ge \widehat N_k+m,
\]
so the shifted prescribed blocks are disjoint.
For $n\ge 1$, define
\[
I_n:=
\begin{cases}
\{a_{n_0+n-1}(x_*)\}, & 1\le n\le \ell,\\
\{b_{n-\widehat N_k+1}^{(r_k)}\}, & \widehat N_k\le n\le \widehat N_k+m-1 \text{ for some }k,\\
E_{n_0+n-1}, & \text{otherwise.}
\end{cases}
\]
Let
\[
\Omega:=\prod_{n\ge 1}I_n,
\qquad
K_0:=\pi(\Omega).
\]
Since every level set $I_n$ is nonempty, the product set $\Omega$ is nonempty.
The first $\ell$ levels are fixed to the initial tail of $x_*$, the levels $\widehat N_k,\dots,\widehat N_k+m-1$ are fixed to the prescribed block $\mathbf b^{(r_k)}$, and all remaining levels vary inside the finite sets $E_{n_0+n-1}$.
By \eqref{eq:anchor-window},
\[
K_0\subset Y_\ell\subset D\cap B(0,r_d).
\]
Fix $y\in K_0$ and choose $\omega\in\Omega$ with $\pi(\omega)=y$.
By Lemma~\ref{lm:fullshift},
\[
a_n(y)=\omega_n\qquad(n\ge 1).
\]
Set $x:=h_{\mathbf c}(y)$.
Since $y\in D$, we have
\[
x\in \ip{\mathbf c},
\qquad
T_d^{n_0-1}(x)=y.
\]
Hence the first $n_0-1$ digits of $x$ are $\mathbf c$, and for every $n\ge 1$,
\[
a_{n_0+n-1}(x)=a_n(y)=\omega_n.
\]
Since $\omega\in\Omega$, we obtain
\[
(a_{N_{r_k}}(x),\dots,a_{N_{r_k}+m-1}(x))=\mathbf b^{(r_k)}
\qquad(k\ge 1).
\]
All digits of $x$ lie in $S$.
On the finite prefix and on the initial reference block this follows from $x_*\in E$; at the variable levels it follows from $E_n\subset S'$; and at the prescribed levels it is part of the hypotheses on the prescribed blocks.
The same decomposition gives the cutoff bounds: the finite prefix and the initial reference block inherit them from $x_*\in E$, the variable levels use $\abs{\alpha}\le f(n)$ for $\alpha\in E_n$, and the prescribed levels use the assumed bounds $\abs{b_j^{(r)}}\le f(N_r+j-1)$.
If $j$ is a variable level, then
\[
a_{n_0+j-1}(x)=\omega_j\in E_{n_0+j-1},
\qquad
\abs{a_{n_0+j-1}(x)}\ge R_{n_0+j-1}.
\]
On the $k$-th prescribed block,
\[
\abs{a_n(x)}\ge \min_{1\le j\le m}\abs{b_j^{(r_k)}}.
\]
Since $R_n\to\infty$ and $\min_{1\le j\le m}\abs{b_j^{(r_k)}}\to\infty$, we obtain
\[
\abs{a_n(x)}\to\infty.
\]
Therefore
\[
h_{\mathbf c}(K_0)\subset E.
\]
By Lemma~\ref{lm:prefix_conformal}, $h_{\mathbf c}$ is bi-Lipschitz on $\overline{B(0,r_d)}$.
Since $K_0\subset \overline{B(0,r_d)}$, it remains to prove
\[
\dim_H(K_0)\ge \frac{t}{2}.
\]
\smallskip\noindent\emph{Pressure.}
For each $n\ge 1$, set
\[
\Psi_n:=\{h_\alpha:\ \alpha\in I_n\}.
\]
Every level map belongs to the autonomous family $\Phi_d^{\ge A_0}$.
Hence Lemma~\ref{lm:trapping} gives the open set condition at each level on $\mathrm{int}(\Delta_d)$, a common extension domain $\widetilde\Delta_d$, and a one-step contraction bound
\[
\sup_{z\in \Delta_d}\abs{Dh_\alpha(z)}\le \gamma<1
\qquad(\alpha\in I_n,\ n\ge 1),
\]
while Lemma~\ref{lm:NT25_geom} yields the required bounded distortion uniformly for all compositions across levels.
Therefore $\Psi:=(\Psi_n)_{n\ge 1}$ is a non-autonomous conformal IFS on $\Delta_d=I_d$ as defined above.
Its limit set is exactly
\[
\Lambda(\Psi)=\pi(\Omega)=K_0,
\]
because the non-autonomous coding map agrees with the restriction of $\pi$ to $\Omega$.
Moreover, $\# I_n=1$ at fixed or prescribed levels, while $\# I_n=\# E_{n_0+n-1}$ at variable levels.
Hence
\[
0\le \frac{1}{n}\log \# I_n\le \frac{1}{n}\log \# E_{n_0+n-1}
=\frac{n_0+n-1}{n}\cdot \frac{\log \# E_{n_0+n-1}}{n_0+n-1}\longrightarrow 0
\]
by Lemma~\ref{lm:tail_select_cutoff}.
Set
\[
q:=\frac{t}{2},
\qquad
\Sigma_j:=\sum_{\alpha\in I_j}(\abs{\alpha}+R_d)^{-t}.
\]
For a word $(\alpha_1,\dots,\alpha_n)\in I_1\times\cdots\times I_n$, the chain rule and
\[
\abs{Dh_\alpha(z)}=\frac{1}{\abs{z+\alpha}^{2}}\ge \frac{1}{(\abs{\alpha}+R_d)^2}
\qquad(z\in I_d)
\]
give
\[
\bigl\|D(h_{\alpha_1}\circ\cdots\circ h_{\alpha_n})\bigr\|_\infty^{q}
\ge \prod_{j=1}^{n}(\abs{\alpha_j}+R_d)^{-t}.
\]
Hence
\[
Z_n(q)\ge \prod_{j=1}^{n}\Sigma_j.
\]
At a variable level $j$,
\[
\Sigma_j
=\sum_{\alpha\in E_{n_0+j-1}}(\abs{\alpha}+R_d)^{-t}
=\sum_{\alpha\in E_{n_0+j-1}}(\abs{\alpha}+R_d)^{\varepsilon}(\abs{\alpha}+R_d)^{-s}
\ge R_{n_0+j-1}^{\varepsilon}\sum_{\alpha\in E_{n_0+j-1}}(\abs{\alpha}+R_d)^{-s}
\ge 2^{\varepsilon}
=e^{8c_0}.
\]
Set
\[
C_{\mathrm{init}}:=\prod_{j=1}^{\ell}(\abs{a_{n_0+j-1}(x_*)}+R_d)^{-t}>0.
\]
For the prescribed block attached to $r_k$,
\[
\prod_{u=0}^{m-1}\Sigma_{\widehat N_k+u}
=\prod_{u=1}^{m}(\abs{b_u^{(r_k)}}+R_d)^{-t}
=\widetilde B_{r_k}^{-t}.
\]
If
\[
k(n):=\#\{i\ge 1:\ \widehat N_i\le n\},
\]
then among the first $n$ levels at most $\ell+mk(n)$ positions are fixed, and the total contribution of the prescribed blocks is bounded by $t\sum_{i=1}^{k(n)}\log \widetilde B_{r_i}$.
Therefore
\[
\log Z_n(q)\ge \log C_{\mathrm{init}}+8c_0\bigl(n-\ell-mk(n)\bigr)-t\sum_{i=1}^{k(n)}\log \widetilde B_{r_i}.
\]
If $k(n)=0$, then
\[
\log Z_n(q)\ge \log C_{\mathrm{init}}+8c_0(n-\ell).
\]
If $k(n)\ge 1$, then \eqref{eq:cumulative-loss-bound} and $\widehat N_{k(n)}\le n$ give
\[
t\sum_{i=1}^{k(n)}\log \widetilde B_{r_i}+8c_0mk(n)
\le 2c_0N_{r_{k(n)}}
\le 2c_0(n+n_0-1),
\]
and hence
\[
\log Z_n(q)\ge 6c_0n-O(1).
\]
In either case,
\[
P_\Psi(q)=\liminf_{n\to\infty}\frac{1}{n}\log Z_n(q)\ge 6c_0>0.
\]
Theorem~\ref{thm:nonaut_bowen} yields
\[
\dim_H(K_0)\ge q=\frac{t}{2}.
\]
Consequently,
\[
\dim_H(E)\ge \dim_H h_{\mathbf c}(K_0)\ge \frac{t}{2}.
\]
Letting $t\uparrow \tau(S)$ and combining with the upper bound gives
\[
\dim_H(E)=\frac{\tau(S)}{2}.
\]
This proves Theorem~\ref{thm:pattern_intro}.
\end{proof}

\section*{Acknowledgement}
We would like to express our gratitude to Professor Seonhee Lim for her guidance throughout this project.

\bibliography{ref}
\bibliographystyle{plainurl}

\end{document}